\documentclass[11pt,twoside,a4paper]{amsart}
\usepackage{amssymb}
\usepackage{amsmath,amscd}
\usepackage{eufrak} 
\usepackage[initials,nobysame,shortalphabetic]{amsrefs}
\usepackage[all,cmtip]{xy}
\usepackage{hyperref}
\date{\today}
\usepackage{geometry}
\usepackage{soul}
\usepackage{xcolor}

\def\la{\langle}
\def\ra{\rangle}

\def\w{\wedge}

\def\dbar{\bar\partial}

\def\R{{\mathbf R}}

\def\C{{\mathbf C}}
\def\w{{\wedge}}
\def\P{{\mathbf P}}

\def\codim{{\rm codim\,}}

\def\U{{\mathcal U}}

\def\be{\begin{equation}}
\def\ee{\end{equation}}

\def\Ok{\mathcal O}

\def\1{{\bf 1}}

\DeclareMathOperator{\Id}{Id}
\DeclareMathOperator{\supp}{supp}

\newcommand{\eps}{\varepsilon}

\newcommand{\FS}{{\textrm{FS}}}

\newcommand{\minus}{\setminus}

\mathcode`l="8000
\begingroup
\makeatletter
\lccode`\~=`\l
\DeclareMathSymbol{\lsb@l}{\mathalpha}{letters}{`l}
\lowercase{\gdef~{\ifnum\the\mathgroup=\m@ne \ell \else \lsb@l \fi}}%
\makeatother
\endgroup

\newtheorem{thm}{Theorem}[section]
\newtheorem{lma}[thm]{Lemma}

\newtheorem{prop}[thm]{Proposition}

\theoremstyle{definition}

\newtheorem{df}[thm]{Definition}

\theoremstyle{remark}

\newtheorem{preremark}[thm]{Remark}
\newtheorem{preex}[thm]{Example}

\newenvironment{remark}{\begin{preremark}}{\qed\end{preremark}}
\newenvironment{ex}{\begin{preex}}{\qed\end{preex}}

\numberwithin{equation}{section}

\title[Chern forms of hermitian metrics with analytic singularities]
{Chern forms of hermitian metrics with analytic singularities on
  vector bundles}

\begin{document}
\date{\today}

\author[Richard L\"ark\"ang \& Hossein Raufi \& Martin Sera \& Elizabeth  Wulcan]
{Richard L\"ark\"ang \& Hossein Raufi \& Martin Sera \& Elizabeth  Wulcan}

\address{R.\ L\"ark\"ang,  H.\ Raufi, E.\ Wulcan, Department of Mathematics\\Chalmers University of Technology and the University of Gothenburg\\S-412 96 
Gothenburg\\SWEDEN}
\email{larkang@chalmers.se, raufi@chalmers.se, wulcan@chalmers.se}
\address{M.\  Sera, Faculty of Engineering\\Kyoto University of Advanced Science\\Kyoto 615-8577\\JAPAN}
\email{sera.martin@kuas.ac.jp}

\subjclass[2010]{32L05, 32U40, 32W20 (14C17, 32U05)}

\thanks{The first and the last author were partially supported by the
  Swedish Research Council. The third author was supported by the
  German Research Foundation (DFG, grant SE 2677/1) and the Knut and Alice Wallenberg Foundation.}

\begin{abstract}
We define Chern and Segre forms, or rather currents, associated with a Griffiths positive singular hermitian
metric $h$ with analytic singularities on a holomorphic vector bundle $E$. 
The currents are constructed as pushforwards of generalized Monge-Amp\`ere products on the
projectivization of $E$. The Chern and Segre currents represent the
Chern and Segre classes of $E$, respectively, and coincide with
the Chern and Segre forms of $E$ and $h$ where $h$ is smooth. 
Moreover, our currents coincide with the Chern and Segre forms constructed by the
first three authors and Ruppenthal in the cases when these are defined. 
\end{abstract}

\maketitle

\section{Introduction}\label{intro}
Singular metrics on line bundles were introduced by Demailly in \cite{Dem2},
and have since developed to be an influential analytic tool in complex algebraic geometry.
In \cite{BP} Berndtsson and P\u aun introduced singular
metrics on vector-bundles in order to prove results about
pseudoeffectivity of relative canonical bundles.
These have been further studied in a series of papers including, e.g.,
\cite{Hos,HPS,R}. In order to develop a theory for
singular metrics on vector bundles it seems crucial to have Chern
forms. 
In the line bundle case the (first) Chern form is a well-defined
current, whereas any attempt to construct Chern forms of singular
metrics on higher rank
bundles seems to involve multiplication of currents.
In \cite{LRRS} the first three authors together with Ruppenthal defined Chern
forms for positive singular metrics on vector bundles 
under a certain natural condition on the dimension of the degeneracy
locus. 
In this paper we define Chern forms without this assumption but for
metrics with so-called analytic singularities. 
To do this we develop a new 
formalism for generalized (mixed) Monge-Amp\`ere operators for
plurisubharmonic functions with analytic singularities
extending the construction in \cite{AW}. 

Let $E$ be a holomorphic vector bundle of rank $r$ over a
complex manifold $X$ of dimension $n$ 
and let 
 $h$ be a smooth hermitian metric on $E$. 
Let $\pi : \P(E)\to X$ be the
projective bundle of lines in $E^*$. Then $h^*$ induces a metric on the
tautological line bundle $\Ok_{\P(E)}(-1)\subset \pi^* E^*$;  let
$e^{-\varphi}$ be the dual metric on $\Ok_{\P(E)}(1)$. 
If $h$ is Griffiths positive, then
$e^{-\varphi}$ is a positive metric, i.e., the local weights $\varphi$
are plurisubharmonic (psh), and the first Chern form 
of $e^{-\varphi}$ is given as $dd^c
\varphi$, where $d^c=(1/4\pi i)(\partial -\dbar)$. 
Note that this is a well-defined global positive $(1,1)$-form. 
Following the ideas in, e.g., \cite{F} one can define the associated
\emph{$k$th Segre form} as
\begin{equation}\label{ny}
s_k(E,h):=(-1)^k\pi_*(dd^c\varphi)^{k+r-1},
\end{equation}
cf.\ \cite[Section~7.1]{M}. Since $\pi$ is a
submersion this is a smooth form of bidegree
$(k,k)$. 
It was proved in \cite[Proposition~6]{M}, see also
\cite[Proposition~1.1]{Div} and \cite[Proposition~3.1]{G}, that \eqref{ny} coincides with the classical
definition of Segre forms, which means that the total Segre form
$s(E,h)=1+s_1(E,h)+s_2(E,h)+\cdots$ is the multiplicative
inverse of the total
Chern form $c(E,h)=1+c_1(E,h)+c_2(E,h)+\cdots$. 
Identifying components of the same bidegree, this can be expressed as 
\begin{equation}\label{identitet}
s_k(E,h) + s_{k-1}(E,h) \wedge c_1(E,h)+\cdots + c_k(E,h)=0, ~~~~~
k=1,2, \ldots ~~ .
\end{equation} 
In particular, \eqref{identitet} holds on cohomology level, i.e., the
total Segre class $s(E)=1+s_1(E)+s_2(E) + \cdots $ is the multiplicative
inverse of the total
Chern class $c(E)=1+c_1(E)+c_2(E)+\cdots $; here $c_k(E)$ and
$s_k(E)$ are the \emph{$k$th Chern} and \emph{Segre classes} of $E$, defined as the de Rham cohomology classes of $c_k(E,h)$
and $s_k(E,h)$, respectively. 

The aim of this paper is to construct Chern and Segre forms, or rather currents, associated
with singular metrics. 
Therefore let $h$ be a Griffiths positive singular metric on $E$ in the sense of
Berndtsson-P\u aun, \cite{BP}, see Section \ref{jensen}; then the induced singular
metric $e^{-\varphi}$ on $\Ok_{\P(E)}$ is positive, cf.\ Proposition ~\ref{storemosse}. 
Our strategy is to mimic the construction \eqref{ny} of Segre forms and use
them to construct Chern and Segre currents. However, in general one cannot take
products of currents and in particular $(dd^c\varphi)^k$ is not always 
well-defined. 

Recall that a psh function $u$ has \emph{analytic singularities} if it is
locally of the form 
\begin{equation}\label{trond}
u=c\log|F|^2+v,
\end{equation}
where $c>0$, $F$ is a tuple of holomorphic functions $f_j$, $|F|^2=\sum|f_j|^2$, and $v$ is 
bounded. 
We say that $h$ has \emph{analytic singularities} if the weights
$\varphi$ are psh with analytic singularities; for a direct definition
in terms of $h$, see Proposition ~\ref{gummistovlar}. 
In \cite{Hos} Hosono constructed a class of examples of singular hermitian
metrics on vector bundles, that in fact have analytic singularities, see Example
~\ref{hosono}.

Given a psh function $u$ with analytic singularities, in \cite{AW} Andersson and the last author defined \emph{generalized
  Monge-Amp\`ere products} $(dd^c u)^m$ recursively as 
\begin{equation}\label{defen}
(dd^c u)^k := dd^c (u \1_{X\setminus Z} (dd^c u)^{k-1}),
\end{equation}
where $Z$ is the unbounded locus of $u$, i.e., locally defined as 
$\{F=0\}$ where $u$ is given by \eqref{trond}; for $u$
of the form $u=\log |F|^2$ the currents \eqref{defen} were defined in
\cite{A}. 
The current 
$(dd^c u)^m$ is positive and closed and of bidegree $(m,m)$. For $m\leq \codim Z$,
it coincides with Bedford-Taylor-Demailly's classically 
defined $(dd^c u)^m$, cf.\ Section ~\ref{classical}. 
If $\alpha$ is a closed smooth $(1,1)$-form, inspired by
\cite[Theorem~1.2]{ABW}, cf.\ Remark ~\ref{massformeln}, we let 
\begin{equation}\label{geten} 
[dd^c u]^m_\alpha:=(dd^c u)^m+\sum_{\ell=0}^{m-1} \alpha^{m-\ell}\1_Z
(dd^c u)^\ell
\end{equation}
for $m\geq 1$ and $[dd^c u]_\alpha^0=1$; 
see ~\eqref{bullen} for a recursive description. 
Note that if $m\leq \codim Z$, then $\1_Z
(dd^c u)^\ell=0$ for $\ell<m$ and thus $[dd^c u]^m_\alpha=(dd^c
u)^m$.

Now assume that $h$ has analytic singularities
and let $\theta$ be the first Chern form of a smooth
metric $e^{-\psi}$ on $\Ok_{\P(E)}(1)$; e.g., $e^{-\psi}$ can be chosen as
the metric on $\Ok_{\P(E)}(1)$ induced by a smooth metric on $E$. 
Since the difference of two local weights $\varphi$ is of the form $\log|f|^2$, where $f$ is a nonvanishing
holomorphic function,  $[dd^c\varphi]_\theta^m$ is a globally defined current
on $\P(E)$, see Section ~\ref{linjen}. 
Inspired by \eqref{ny} we define
\begin{equation}\label{dansk}
s_k(E,h,\theta):=(-1)^k\pi_*[dd^c\varphi]_\theta^{k+r-1}.
\end{equation}
If the $\varphi$ are smooth, then clearly $s_k(E,h,\theta)$ coincides with
$s_k(E,h)$ defined by \eqref{ny}.

To construct Chern currents we need to define products of this kind of
currents. Let $E_1,\ldots, E_t$ be disjoint copies of $E$ and let $\pi:Y\to X$ be the
fiber product $Y=\P(E_t) \times_X \cdots \times_X \P(E_1).$
Let $\varphi_j$ and $\theta_j$ denote the pullbacks to $Y$ of the
metric and the form on $\P(E_j)$ corresponding to $\varphi$ and
$\theta$, respectively. 
By extending ideas in \cite{AW} and \cite{ASWY} we give meaning to products 
\[
[dd^c \varphi_t]_{\theta_t}^{m_t}\wedge\cdots\wedge [dd^c \varphi_1]_{\theta_1}^{m_1}
\]
on $Y$, see Sections ~\ref{general} and ~\ref{linjen}. 
Next for $k_j\geq 1$ we define
\begin{equation}\label{korsbar}
s_{k_t}(E,h,\theta)\wedge\cdots\wedge s_{k_1}(E,h,\theta)
:=
(-1)^k\pi_* \big ( [dd^c
\varphi_t]_{\theta_t}^{k_t+r-1}\wedge\cdots\wedge [dd^c
\varphi_1]_{\theta_1}^{k_1+r-1}
\big ), 
\end{equation}
see Section ~\ref{plommon}; 
here and throughout
$k:= k_1+\cdots + k_t$. 
If $h$, and thus $\varphi$, is smooth, then \eqref{korsbar} just
coincides with the product $s_{k_t}(E,h)\wedge\cdots\wedge
s_{k_1}(E,h)$ of smooth Segre forms, cf.\ Section ~\ref{sakura}. 
The currents \eqref{korsbar} are in general not commutative in the
factors $s_{k_j}(E,h,\theta)$, see Example ~\ref{martinex}. 
Now we can recursively define Chern currents $c_k(E,h,\theta)$ using the identities
\eqref{identitet}, ~i.e., 
\begin{eqnarray}
c_1(E,h, \theta)&:=&-s_1(E,h,\theta),\nonumber\\
c_2(E,h,\theta)&:=&s_1(E,h,\theta)^2-s_2(E,h,\theta),\nonumber\\
& \vdots &\nonumber \\
c_k(E,h,\theta) &:=& \sum_{k_1+\cdots +k_t=k}(-1)^ts_{k_t}(E,h,\theta)\wedge\cdots\wedge
s_{k_1}(E,h,\theta).\label{koko}
\end{eqnarray}

\begin{thm}\label{thmA}
Let $h$ be a Griffiths positive hermitian metric with analytic singularities on
the holomorphic vector bundle $E\to X$ over a complex manifold $X$ of
dimension $n$ and let $\theta$ be the first Chern form of a smooth metric
on $\Ok_{\P(E)}(1)$. 
Then for $k=1,2, \ldots$ $c_k(E,h,\theta)$ and $s_k(E,h,\theta)$ defined by \eqref{dansk}
and \eqref{koko}, respectively, are closed normal 
 $(k,k)$-currents; more precisely they are locally differences of
closed positive currents.
Moreover
\begin{enumerate}
\item\label{forsta}
$c_k(E,h,\theta)$ and $s_k(E,h,\theta)$ represent the $k$th Chern and
Segre classes $c_k(E)$ and $s_k(E)$ of $E$, respectively, as de Rham cohomology classes of currents,
\item\label{andra}
$c_k(E,h,\theta)$ and $s_k(E,h,\theta)$ coincide with the Chern and Segre forms $c_k(E,h)$ and
$s_k(E,h)$, respectively, where $h$ is smooth, 
\item\label{tredje}
the Lelong numbers of $c_k(E,h,\theta)$ and $s_k(E,h,\theta)$ at each $x \in X$ are independent of
$\theta$. 
\end{enumerate}
\end{thm}

Note that Lelong numbers of $c_k(E,h,\theta)$ and $s_k(E,h,\theta)$
are well-defined since 
the currents are locally differences of closed positive currents,
cf.\ Section \ref{lelongsection}.

Assume that the unbounded locus of $\log
\det h^*$ is contained in a variety $V$ of pure codimension $p$. 
Then for $k\leq p$, $c_k(E,h,\theta)$ and
$s_k(E,h,\theta)$ are independent of $\theta$, see Section ~\ref{losi}. In general, however, $c_k(E,h,\theta)$ and
$s_k(E,h,\theta)$ do depend on $\theta$, cf.\ Examples ~\ref{kass}, ~\ref{linjalen}
and \ref{martinex}. 

In \cite{LRRS} the first three authors together with Ruppenthal showed
that if $h$ is a singular hermitian metric (not necessarily with
analytic singularities) such that the unbounded locus of $\log
\det h^*$ is contained in a variety $V$ of pure codimension $p$, then 
for $k_1+\cdots+k_t\leq p$ one can give meaning to currents $s_{k_t}(E,h)\wedge\cdots\wedge
s_{k_1}(E,h)$ as limits of $s_{k_t}(E,h_{\varepsilon_t})\wedge\cdots\wedge
s_{k_1}(E,h_{\varepsilon_1})$, where $h_{\varepsilon_j}$ are smooth metrics
approximating $h$, see Section ~\ref{jamfora}. Analogously to \eqref{koko} one can
then define Chern
currents $c_k(E,h)$ for $k\leq p$. We should remark that this
construction cannot be extended to general $k$, see
Example \ref{kass}.

\begin{thm}\label{thmB}
Let $h$ be a Griffiths positive hermitian metric with analytic singularities on
the holomorphic vector bundle $E\to X$ over a complex manifold $X$ 
and let $\theta$ be the first Chern form of a smooth
metric on $\Ok_{\P(E)}(1)$. 
Assume that the unbounded locus of $\log \det h^*$ is contained in a variety $V\subset X$. 
Then for $k_1+\cdots + k_t\leq \codim V$, 
\begin{equation}\label{jason} 
s_{k_t}(E,h,\theta)\wedge\cdots\wedge s_{k_1}(E,h,\theta)
=
s_{k_t}(E,h)\wedge\cdots\wedge s_{k_1}(E,h).
\end{equation}
In particular $c_k(E,h,\theta)=c_k(E,h)$ 
for $k\leq \codim V$. 
\end{thm}

The paper is organized as follows. In Section ~\ref{mma} we give some
background on currents and classical Monge-Amp\`ere
products. In Section ~\ref{general} we introduce mixed Monge-Amp\`ere
products of psh functions with analytic singularities generalizing
\eqref{defen}. Next in Sections ~\ref{linjen} and ~\ref{jensen} we study
metrics with analytic singularities on line bundles and vector
bundles, respectively. The proofs of Theorems ~\ref{thmA} and ~\ref{thmB}
occupy Sections ~\ref{konstruera} and ~\ref{jamfora},
respectively. Finally, in Section \ref{losi} we conclude with some
examples and remarks.

\section{Preliminaries}\label{mma}
Let us first recall some results on (closed positive) currents. 
If $\pi:\widetilde X\to X$ is a proper map, $\mu$ is a
current on $\widetilde X$, and $\alpha$ is a smooth form on $X$, then
we have the projection formula 
\begin{equation}\label{daghem}
\alpha\wedge \pi_*\mu =\pi_*(\pi^*\alpha\wedge \mu).
\end{equation}
Moreover, if $p$ is a proper submersion, $\mu$ is a current on $X$, 
and $\alpha$ is a smooth form on $\widetilde X$, then
\begin{equation}\label{daghem2}
p_* \alpha \wedge \mu =p_*(\alpha\wedge p^*\mu).
\end{equation}

The \emph{Poincar\'e-Lelong formula} asserts that if $f$ is a
holomorphic function defining a divisor $D$, then 
\begin{equation*}
dd^c \log |f|^2=[D],
\end{equation*}
where $[D]$ is the current of integration along the divisor of $f$.

Given a subset $A\subset X$, let $\1_A$ denote the characteristic
function of $A$. 
If $Z\subset X$ is a subvariety and $T$ is a closed positive current
on $X$,
then the Skoda-El Mir theorem asserts that $\1_{X \setminus Z} T$ 
is again positive and closed. 
It follows that if $U\subset X$ is any constructible set,\!\!%
\footnote{{Simple examples of constructible sets are $V\minus W$ or $(X\minus V)\cup W$ for analytic sets $V$ and $W$ in $X$.}}
i.e., a set in the Boolean algebra generated by Zariski open sets in $X$,
then also $\1_UT$ is positive and closed.
Note that if $U_1$ and $U_2$ are two constructible sets in $X$, then
\begin{equation}\label{frisk}
    \1_{U_1 \cap U_2} T = \1_{U_1} \1_{U_2} T.
\end{equation}
Also, note that if $\chi_\epsilon$ is any sequence of bounded functions such
that $\chi_\varepsilon \to \1_U$ pointwise, then
by dominated convergence, $\1_U T = \lim_\varepsilon \chi_\epsilon
T$. 
It follows that if $\pi$ is as above, then 
\begin{equation}\label{bla}
\1_U\pi_*T = \pi_*(\1_{\pi^{-1} U}T).
\end{equation}
Moreover, if $Z\subset X$ is a subvariety (locally) defined by a holomorphic
tuple $F$, then $\1_{X\setminus Z} T$ equals
the limit of $|F|^{2\lambda} T$ as $\lambda\rightarrow 0^+$.

Finally recall that a closed positive (or normal) current of bidegree $(k,k)$ on $X$ that has
support on a subvariety of $X$ of codimension $>k$ vanishes. We refer
to this as the \emph{dimension principle}. In particular, if $W, Z \subset
X$ are subvarieties such that $W$ is of pure codimension $p$ and
$\codim (Z\cap W)>p$, then 
\begin{equation}\label{blabla}
\1_Z[W]=0.
\end{equation}

\subsection{Classical Bedford-Taylor-Demailly Monge-Amp\`ere products}\label{classical} 

Let $u_1,\ldots, u_m$ be locally bounded psh functions on a complex manifold
$X$ and let $T$ be a closed positive current on $X$. 
The classical Bedford-Taylor theory asserts that one can define a
closed positive current 
\begin{equation*}
dd^c u_m\wedge\cdots\wedge dd^c u_1\wedge T
\end{equation*}
recursively as 
\begin{equation}\label{slott}
dd^c u_{k}\wedge\cdots\wedge dd^c u_1\wedge T:=
dd^c\big (u_{k} dd^c u_{k-1}\wedge\cdots\wedge dd^c u_1\wedge T\big ), 
\end{equation}
for $k=1,\ldots, m$. 
This current satisfies the following version of the
Chern-Levine-Nirenberg inequalities, see \cite[Chapter~III,
Propositions~3.11 and~4.6]{Dem}. 

\begin{prop}\label{cln}
Given compacts $L \subset\subset K$ there is a constant $C_{K,L}$
such that for all closed positive currents $T$ and psh functions
$v,u_1,\ldots, u_m$, where $u_1,\ldots, u_m$ are locally
bounded,
\[
||v dd^c u_m\wedge\cdots\wedge dd^c u_1\wedge T||_L\leq C_{K,L} 
||v||_{L^1(K)} ||u_1||_{L^\infty(K)} \cdots ||u_m||_{L^\infty(K)}
||T||_K.
\]
\end{prop}
Here $\|S\|_K$ is the \emph{mass semi-norm} of the order zero current
$S$ with respect to the compact set $K$, see \cite[Chapter
~III.3]{Dem}.

Recall that the \emph{unbounded locus} $L(u)$ of a psh function $u$
is the set of points $x\in X$ such that $u$ is unbounded in every
neighborhood of $x$. 
Note that if $u$ has analytic singularities, then $L(u)$ is an analytic set, locally defined by
$F=0$ where $u$ is given by \eqref{trond}.  
Demailly \cite{D} extended the definition \eqref{slott} to the case when the
intersection of the unbounded loci of the $u_j$ is small in a certain
sense. 
 The following is a simple corollary of 
\cite[Chapter ~III, Theorem~4.5]{Dem}.

\begin{prop}\label{propB}
Let $u_1,\ldots, u_m$ be psh functions on a complex manifold $X$ such that the
unbounded locus $L(u_j)$ is contained in analytic set $Z_j\subset X$ for
each $j$. Moreover, let $T$ be a closed positive current of
bidegree $(p,p)$ with support contained in an analytic set $W\subset
X$. 
Assume that 
\begin{equation*}
\codim(Z_{i_1}\cap \cdots \cap Z_{i_\ell}\cap W)\geq \ell + p
\end{equation*}
for all choices of $1\leq i_1<\cdots <i_\ell\leq m$. 
Then 
$u_m dd^c u_{m-1} \wedge\cdots\wedge dd^c u_1\wedge T$ and 
$dd^c u_m\wedge\cdots\wedge dd^c u_1\wedge T$ are well-defined and
have locally finite mass. The latter is a closed positive current. 
\end{prop}

These products satisfy the following continuity properties, see, e.g.,
\cite[Chapter~III, Proposition~4.9]{Dem}.

\begin{prop}\label{addis}
Let $u_j$ and $T$ be as in Proposition \ref{propB}. 
If $u_j^{(\iota)}$ are sequences of psh
functions decreasing to $u_j$, then 
\begin{eqnarray*}
&u_m^{(\iota)}dd^c u_{m-1}^{(\iota)}\wedge\cdots\wedge dd^c u_1^{(\iota)}\wedge T
\to 
u_m dd^c u_{m-1}\wedge\cdots\wedge dd^c u_1\wedge T
\\
& dd^c u_m^{(\iota)}\wedge\cdots\wedge dd^c u_1^{(\iota)}\wedge T
\to 
 dd^c u_m\wedge\cdots\wedge dd^c u_1\wedge T.
\end{eqnarray*}
\end{prop}

We have the following generalization of \eqref{blabla}. 
\begin{lma}\label{lemmaA}
Assume that $W\subset X$ is a subvariety of pure codimension $p$, and
assume that $Z\subset X$ is {a} subvariety, such that $\codim_X (Z\cap
W)>p$. Moreover, assume that $b_1,\ldots, b_\ell$ are locally bounded
psh functions. Then 
\begin{equation}\label{malva}
\1_Z dd^c b_\ell\wedge \cdots \wedge dd^c b_1 \wedge [W]=0.
\end{equation} 
\end{lma}

\begin{remark}\label{dimen}
Assume that $W\subset X$ is a subvariety and $U\subset X$ is a
constructible set. 
Then it follows from Lemma \ref{lemmaA} that 
\begin{equation}\label{planet}
\1_U dd^c b_\ell\wedge \cdots\wedge dd^c b_1 \wedge [W]=
dd^c b_\ell\wedge \cdots\wedge dd^c b_1 \wedge \1_U [W].
\end{equation}
To prove \eqref{planet} we may assume that $U$
is a subset of $W$ and that $W$ is irreducible. We then need to prove
that 
\begin{equation*}
\1_U dd^c b_\ell\wedge \cdots\wedge dd^c b_1 \wedge [W]=
dd^c b_\ell\wedge \cdots\wedge dd^c b_1 \wedge [W]
\end{equation*}
if $U$ is dense in $W$ and 
$\1_U dd^c b_\ell\wedge \cdots\wedge dd^c b_1 \wedge [W]=0$ otherwise. 

To see this, note on the one hand, that if $U$ is not dense in $W$, then, since $W$ is irreducible,
$\overline U$ is a subvariety of $W$ of positive codimension.
Thus 
\begin{equation*}
\1_U dd^c b_\ell\wedge \cdots\wedge dd^c b_1 \wedge [W]=
\1_U\1_{\overline U} dd^c b_\ell\wedge \cdots\wedge dd^c b_1 \wedge [W]=0
\end{equation*}
in view of \eqref{frisk} and Lemma ~\ref{lemmaA}. 
On the other hand, if $U$ is dense in $W$, i.e., $\overline U=W$, since  $\overline{\overline U\setminus U}$ is a
subvariety of $W$ of positive codimension, then again using Lemma
~\ref{lemmaA}, note that 
\begin{multline*}
\1_U dd^c b_\ell\wedge \cdots\wedge dd^c b_1 \wedge [W]
=
\1_{W} dd^c b_\ell\wedge \cdots\wedge dd^c b_1 \wedge [W]
-\\
\1_{\overline U\setminus U} \1_{\overline{\overline U\setminus U}}
dd^c b_\ell\wedge \cdots\wedge dd^c b_1 \wedge [W] =
dd^c b_\ell\wedge \cdots\wedge dd^c b_1 \wedge [W]. 
\end{multline*}
\end{remark}

\begin{proof}[Proof of Lemma ~\ref{lemmaA}]
We may assume that $X$ is connected. Let us first assume that
$W=X$ so that $[W]=1$ and $Z\subset X$ is a subvariety of positive codimension. Then the lemma follows by a small
modification of the proof of Corollary ~3.3 in \cite{AW}: 
It is enough to consider the case when $Z$ is smooth.
The general case then follows by stratification. Since it is a local statement,
we may choose coordinates $z$ so that $Z=\{z_1=\cdots = z_q =0\}$,
where $q=\codim Z$. In
view of \eqref{frisk} it is enough to prove that $\1_{\{z_1=0\}}dd^c b_\ell\wedge \cdots \wedge dd^c b_1 =0$. Notice that in a set
$|z_1|\le r, |(z_2,\ldots, z_n)|\le r'$, we have that
$\1_{\{z_1=0\}} (dd^c b_\ell\wedge\cdots \wedge dd^c b_1)$ is the 
limit of
$$
-(|z_1|^{2\lambda}-1)(dd^c b_\ell\wedge\cdots \wedge dd^c b_1)
$$
as $\lambda\rightarrow 0^+$,
cf.\ the beginning of this section. 
Since $|z_1|^{2\lambda}-1$ is psh, \eqref{malva} follows from
Proposition ~\ref{cln} 
since the total mass of 
$|z_1|^{2\lambda}-1$ tends to $0$ when $\lambda\to 0^+$.

\smallskip

Next, let us assume that $W$ is smooth. 
We claim that then 
\begin{equation}\label{ros}
dd^cb_\ell\wedge\cdots\wedge dd^cb_1\wedge [W]=
i_*(dd^c i^* b_\ell\wedge\cdots\wedge dd^ci^* b_1), 
\end{equation}
where $i$ is the inclusion $i:W\to X$. 
Taking this for granted, since $\dim (Z\cap
W)>p$ it follows that $i^{-1}Z$ is a proper subvariety of $W$ and thus 
$\1_{i^{-1}Z} dd^c i^*b_\ell \wedge \cdots \wedge {dd^c}
i^*b_1$ vanishes by the argument above, and thus in view of
\eqref{bla}, 
\[
\1_Z dd^cb_\ell\wedge\cdots\wedge dd^cb_1\wedge [W]
=
i_*\big (\1_{i^{-1}Z} dd^c i^*b_\ell \wedge \cdots \wedge
dd^c i^*b_1\big )=0.
\]

It is clear that \eqref{ros} holds if the $b_j$ are smooth. For general  
locally bounded psh $b_j$ let $b^{(\iota)}_j$ be sequences 
of smooth psh functions decreasing to $b_j$. Then 
\eqref{ros} follows from the smooth case and Proposition ~\ref{addis}.

\smallskip

For the general case, let $\pi:\widetilde X\to X$ be an embedded
resolution of $W$ in $X$. In particular, 
$\widetilde W:=\overline{\pi^{-1}W_{\text{reg}}}$ is a smooth manifold
of pure codimension $p$ in $\widetilde X$ and $\pi$ is a
biholomorphism outside a hypersurface $H\subset \widetilde X$, such
$\pi (H\cap \widetilde W)$ has codimension $>p$ in $X$. 
It follows that outside $\pi (H\cap \widetilde W)$, $\pi_*[\widetilde
W]=[W]$, and by the dimension principle this holds everywhere on
$X$.

If $b_j$ are smooth, then in view of \eqref{daghem} 
\begin{equation}\label{kassa}
dd^c b_\ell\wedge\cdots\wedge dd^c b_1\wedge [W]=
\pi_*\big (dd^c \pi^* b_\ell\wedge\cdots\wedge dd^c
\pi^* b_1\wedge [\widetilde W]\big ).
\end{equation}
For general $b_j$ \eqref{kassa} follows by approximating by sequences
of smooth psh functions converging to $b_j$ using Proposition
~\ref{addis}. 

Next let $\widetilde Z=\pi^{-1} Z$. Then $\widetilde
Z\cap \widetilde W\subset \widetilde X$ is a subvariety of codimension
$>p$. Indeed, for each connected component $\widetilde W_j$ of
$\widetilde W$, $\widetilde Z\cap \widetilde W_j$ is a subvariety of
$\widetilde W_j$. Assume that $\widetilde Z\cap \widetilde W_j=
\widetilde W_j$ for some $j$. Then 
\[
Z\cap W\supset \pi(\widetilde Z\cap \widetilde W) \supset \pi
(\widetilde W_j);
\]
however $\pi(\widetilde W_j)$ has codimension $p$, which contradicts that
$Z\cap W$ has codimension $>p$. Thus $\widetilde Z\cap \widetilde W_j$
is a proper subvariety of $\widetilde W_j$ for each $j$. Thus, as
proved above, 
\[
\1_{\widetilde Z} dd^c \pi^* b_\ell\wedge\cdots\wedge dd^c
\pi^* b_1\wedge [\widetilde W] =0
\]
and, in view of \eqref{bla}, we conclude that 
\[
\1_Z dd^c b_\ell\wedge\cdots\wedge dd^c b_1\wedge [W] = 
\pi_*\big (\1_{\widetilde Z} dd^c \pi^* b_\ell\wedge\cdots\wedge dd^c
\pi^* b_1\wedge [\widetilde W] \big )=0.
\]
\end{proof}

\section{Generalized mixed Monge-Amp\`ere products}\label{general}

Assume 
that $u_1,\ldots, u_m$ are psh functions with analytic
singularities on a complex manifold $X$ with unbounded loci $Z_1,\ldots,
Z_m$, respectively. 
Moreover assume that $U_1,\ldots,
U_m\subset X$ are constructible sets contained in $X\setminus Z_1,
\ldots, X\setminus Z_m$, respectively. 
Inspired by \cite[Section~4]{AW} 
we consider currents
\begin{equation}\label{blomma2}
dd^c u_m\1_{U_m}\wedge\cdots\wedge dd^c u_1 \1_{U_1}
\end{equation}
defined recursively as 
\begin{equation}\label{insta2}
dd^c u_{k}\1_{U_k}\wedge \cdots\wedge dd^c u_1\1_{U_1}:=
dd^c\big
(u_{k}\1_{U_{k}}dd^c u_{k-1}\1_{U_{k-1}}\wedge\cdots\wedge
dd^c u_1\1_{U_1}\big ) 
\end{equation}
for $k=1,\ldots, m$. 
In particular, if $u_j=u$ and $U_j=X\setminus Z_j$ for all $j$, then \eqref{insta2}
coincides with \eqref{defen}. 
For aesthetic reasons, and to emphasize that 
$U_j$ is associated with $u_j$, we 
choose to write \eqref{blomma2} rather than
\[
dd^c u_m\wedge \1_{U_m}dd^c u_{m-1}\cdots dd^c u_2\wedge \1_{U_2}dd^c
u_1\1_{U_1}. 
\]
We say that a current of the form $\1_U
dd^c u_m\1_{U_m}\wedge\cdots\wedge dd^c u_1 \1_{U_1}$, where $U$ is a
constructible set, is a 
\emph{(closed
  positive) current with analytic singularities}. We also include currents
\eqref{blomma2} with no factor $dd^c u_j\1_{U_j}$; in other words
$\1_U$ is also a current with analytic singularities. 
For $u_j$ of the form $\log |F_j|^2$ currents like \eqref{blomma2} 
were defined in \cite[Section~5]{ASWY}. 
Note that $dd^c u_{m}\1_{U_{m}}\wedge \cdots\wedge
dd^c u_1\1_{U_1}$ vanishes unless $U_1$ is dense in (at least one
connected component of) $X$.

Proposition ~\ref{liseberg} below asserts that this definition makes sense
and that the currents $dd^c u_{k}\1_{U_k}\wedge\cdots\wedge
dd^c u_1\1_{U_1}$ are
positive and closed. Moreover, Proposition ~\ref{samma} asserts that
\begin{equation}\label{blomma}
dd^c u_m\1_{X\setminus Z_{m}}\wedge\cdots\wedge dd^c u_1\1_{X\setminus Z_1}
\end{equation}
coincides with $dd^c u_m\wedge\cdots\wedge dd^c u_1$,
when this current is well-defined, cf.\ Proposition ~\ref{propB}. 
It is therefore tempting to think of \eqref{blomma} as a generalized mixed
Monge-Amp\`ere product, and just denote it by
$dd^c u_m\wedge\cdots\wedge dd^c u_1$. 
The following example shows, however, that this ``product''
lacks some properties one would
naturally ask for of a product. In particular, it is not additive in
any factor except the right-most one 
nor commutative.

\begin{ex}
Let $u_1=\log|z_1|^2$ and $u_2=\log|z_1z_2|^2$ in $X=\C^2$. Then
$u_1$ and $u_2$ are psh with analytic singularities with
unbounded loci $Z_1=\{z_1=0\}$ and $Z_2=\{z_1=0\}\cup\{z_2=0\}$,
respectively. 
In view of the Poincar\'e-Lelong formula 
it follows that 
\begin{equation*}
dd^c u_2\1_{X\setminus Z_2}\wedge dd^c u_1\1_{X\setminus Z_1}=
dd^c\big (u_2\1_{X\setminus Z_2}[z_1=0]\big)=0
\end{equation*}
but
\begin{equation*}
dd^c u_1\1_{X\setminus Z_1}\wedge dd^c u_2 \1_{X\setminus Z_2}= dd^c\big
(u_1\1_{X\setminus Z_1}
([z_1=0]+[z_2=0])\big)=[z_1=0]\wedge [z_2=0]=[0], 
\end{equation*}
so that $dd^c u_2\1_{X\setminus Z_2}\wedge dd^c u_1\1_{X\setminus Z_1}$ is not commutative in the factors
$dd^c u_j\1_{X\setminus Z_j}$. 

Moreover, let $u_3=\log|z_2|^2$. Then $u_3$ is psh with analytic
singularities with unbounded locus $Z_3=\{z_2=0\}$. Note that $u_2=u_1+u_3$. 
Now 
\begin{multline*}
dd^c u_1\1_{X\setminus Z_1}\wedge dd^c u_1\1_{X\setminus Z_1} + dd^c u_3\1_{X\setminus Z_3}\wedge dd^c u_1\1_{X\setminus Z_1}
=\\
0+ [z_2=0]\wedge [z_1=0]= [0]\neq dd^c u_2\1_{X\setminus Z_2}\wedge
dd^c u_1\1_{X\setminus Z_1}. 
\end{multline*}
\end{ex}

\begin{prop}\label{liseberg}
Let $u_j$ and $U_j$ be as above. 
Assume that $dd^c u_k\1_{U_k}\wedge\cdots\wedge dd^c u_1\1_{U_1}$ is
inductively defined via \eqref{insta2}. 
Let $u_{k+1}^{(\iota)}$ be a sequence of 
smooth psh functions in $X$ decreasing to $u_{k+1}$. Then
$$
u_{k+1}\1_{U_{k+1}} dd^c u_k\1_{U_k}\wedge\cdots\wedge
dd^c u_1\1_{U_1} 
:=
\lim _\iota u^{(\iota)}_{k+1}\1_{U_{k+1}} dd^c u_k\1_{U_k}\wedge\cdots\wedge dd^c u_1\1_{U_1}
$$
has locally finite mass and does not depend on the choice of sequence
$u^{(\iota)}_{k+1}$. 
Moreover 
$dd^c u_{k+1}\1_{U_{k+1}}\wedge \cdots\wedge dd^c u_1\1_{U_1}$, defined by
\eqref{insta2},  
is positive and closed. 
\end{prop}

\begin{remark}\label{stress}
Note that Proposition ~\ref{liseberg} asserts that if $T$ is {a}
current with analytic singularities, $u$ is a psh function with
analytic singularities with unbounded locus $Z$, and $U$ is a
constructible set contained in $X\setminus Z$, then 
\[
dd^c u\wedge \1_U T:=dd^c(u\1_UT)
\]
 is a well-defined current with analytic singularities. 
\end{remark}

The proof is a generalization of the proof of Proposition ~4.1 in
\cite{AW}.

\begin{proof}
Since the statement is local we may assume that
$u_j=\log|F_j|^2+v_j$. Moreover without loss of generality we may
assume that $X$ is connected and that $U_1$ is dense in $X$, and thus $\1_{U_1}=1$ as a
distribution. Indeed, otherwise $dd^c u_k\1_{U_k}\wedge\cdots\wedge
dd^c u_1\1_{U_1}=0$. 

Let $\pi\colon \widetilde X\to X$ be a smooth modification such that 
locally on $\widetilde X$, $\pi^*F_j=f_jf_j'$, where 
$f_j$ is a holomorphic function and 
$f_j'$ is a nonvanishing tuple of holomorphic functions, for each
$j$. 
It follows that $\pi^*u_j=\log|f_j|^2+b_j$, where $b_j:=\log|f_j'|^2+\pi^*v_j$ is
psh and locally bounded, cf.\ the proof of Proposition~4.1 in
\cite{AW}. 
Note that for two different local representations, the $f_j$ differ by
a nonvanishing holomorphic factor, and thus
the $b_j$ differ by a pluriharmonic term. Therefore the local divisors
$\{f_j=0\}$ define a global divisor $D_j$ on $\widetilde X$, such that
$\pi^{-1} Z_j=|D_j|$ and
moreover the currents $dd^cb_j$ define a global positive current on $\widetilde
X$. 
By the Poincar\'e-Lelong formula 
\begin{equation*}
dd^c\pi^* u_j=[D_j]+dd^c b_j.
\end{equation*}

\smallskip
Let $u_{1}^{(\iota)}$ be a sequence of smooth psh functions
decreasing to  $u_1$.  Since $\pi$ is a modification 
$$
dd^c u_1^{(\iota)}=\pi_*(dd^c\pi^* u_1^{(\iota)})\to \pi_*\big(dd^c \pi^* u_1\big)
=\pi_*\big([D_1]+dd^c b_1\big)
$$
and it follows that
\begin{equation*}
dd^c u_1=\pi_*\big([D_1]+dd^c b_1\big).
\end{equation*}

Let us now assume that we have proved that 
$u_k\1_{U_k}dd^c u_{k-1}\1_{U_{k-1}}\wedge \cdots\wedge dd^c
u_1\1_{U_1}$ is well-defined with the desired properties 
and that $dd^c u_k \1_{U_k}\wedge\cdots\wedge dd^c u_1 \1_{U_1}$ is the pushforward of 
\begin{equation}\label{gunnebo}
 \sum_{I=\{i_1,\ldots, i_s\}\subset\{1,\ldots, k\}} 
dd^cb_I \wedge [V_I], 
\end{equation}
where 
$dd^c b_{I}=dd^c b_{i_s}\wedge \cdots\wedge dd^c
b_{i_1}$ and $V_I$ are analytic cycles\footnote{{formal linear combinations of irreducible analytic sets}}
of pure codimension $k-s$ on
$\widetilde X$.
Since the $b_j$ are determined up to addition by a pluriharmonic
term, each 
$dd^cb_I \wedge [V_I]$ is a globally defined current
on $\widetilde X$, cf.\ Section ~\ref{classical}. 

We will prove that:  
\smallskip

\noindent (i) \emph{$u_{k+1}\1_{U_{k+1}} dd^c u_k \1_{U_k}\wedge\cdots\wedge
  dd^c u_1 \1_{U_1} :=\lim_\iota u_{k+1}^{(\iota)}\1_{U_{k+1}}
  dd^c u_k \1_{U_k}\wedge\cdots\wedge dd^c u_1 \1_{U_1} $ has locally finite
mass and is
independent of $u_{k+1}^{(\iota)}$,}

\smallskip

\noindent (ii) \emph{the current 
\begin{equation*}
dd^c u_{k+1}\1_{U_{k+1}}\wedge\cdots\wedge
dd^c u_1 \1_{U_1}:=dd^c\big(u_{k+1} \1_{U_{k+1}} dd^c
u_k \1_{U_k}\wedge\cdots\wedge dd^c u_1 \1_{U_1}\big),
\end{equation*}
is the pushforward of a current of the form \eqref{gunnebo}.}

\smallskip
\noindent As soon as (i) and (ii) are verified, the proposition
follows by induction. 
\smallskip

Note that $\widetilde U_j:= \pi^{-1}U_j$ is a constructible set in
$\widetilde X$. 
Let us consider one summand 
$dd^c b_I \wedge [V_I]$
in \eqref{gunnebo}. 
Let $V_I'$ be the union of the irreducible components $V_I^j$ of $V_I$
such that $\widetilde U_{k+1}\cap V_I^j$ is dense in $V_I^j$. Then in
view of Remark ~\ref{dimen} 
\begin{equation*}
\1_{\widetilde U_{k+1}}  
dd^c b_I\wedge [V_I]=
dd^c b_I\wedge [V_I'].
\end{equation*}

We claim that $|D_{k+1}|\cap |V_I'|$ has positive codimension in
$|V_I'|$. To see this let $V_I^j$ be an irreducible component of
$|V_I'|$. Then either $V_I^j\subset |D_{k+1}|$ or $|D_{k+1}|\cap V_I^j$ has
positive codimension in $V_I^j$. However $V_I^j$ cannot be contained
in $|D_{k+1}|$ since $\widetilde U_{k+1}\cap V_I^j\subset (\widetilde
X\setminus |D_{k+1}|)\cap V_I^j$ is dense in $V_I^j$; this
proves the claim. 

Since $\codim (|V_I'|\cap |D_{k+1}|)>\codim |V_I'|$, 
by Proposition
~\ref{propB},  
$\pi^*u_{k+1}
dd^c b_I \wedge [V_I']$
has locally finite mass and by Proposition ~\ref{addis}, 
\[
\pi^* u_{k+1}^{(\iota)} dd^cb_I 
\wedge [V_I'] \to 
\pi^* u_{k+1} dd^c b_I 
\wedge [V_I'] 
\]
if $u_{k+1}^{(\iota)}$ is any  sequence of psh
functions decreasing to $u_{k+1}$. 
If $u_{k+1}^{(\iota)}$ are smooth, using \eqref{daghem}, that $\pi^{-1}
U_{k+1}=\widetilde U_{k+1}$, and \eqref{bla}, we get 
\[
u_{k+1}^{(\iota)} \1_{U_{k+1}}dd^c u_k\1_{U_k}\wedge
\cdots \wedge dd^c u_1\1_{U_1}
=
\pi_* \big (\pi^* u_{k+1}^{(\iota)}  \sum_I  dd^c b_I 
\wedge [V_I'] \big ). 
\]
Proposition ~\ref{addis} then yields
\begin{equation}\label{flyga}
u_{k+1} \1_{U_{k+1}} dd^c u_k\1_{U_k}\wedge
\cdots \wedge dd^c u_1 \1_{U_1}
=
\pi_* \big (\pi^* u_{k+1} \sum_I  dd^c b_I 
\wedge [V_I'] \big ). 
\end{equation}
Clearly the right hand side has locally finite mass, and thus (i) is verified.

We now consider (ii). 
Since the local representation $\pi^*F_j=f_jf_j'$ is determined up to
multiplication by a pluriharmonic factor, it follows, e.g., from Proposition ~\ref{addis} that 
\begin{equation}\label{mayo}
dd^c \pi^*u_{k+1} \wedge 
dd^c b_I \wedge [V_I'] 
=
dd^c \log|f_{k+1}|^2
\wedge dd^c b_I 
\wedge [V_I'] 
+
dd^c b_{k+1} \wedge dd^c b_I 
\wedge [V_I'], 
\end{equation}
cf.\ the proof of Proposition ~4.1 in \cite{AW}. The second term is
clearly of the desired form. 
Since $D_{k+1}$ and $V_I'$ intersect properly, $[D_{k+1}]\wedge[V_I']$
is the current of integration of a cycle $V$ of pure codimension
$k-s+1$, cf.\ \cite[Chapter~III, Proposition~4.12]{Dem}. 
Thus 
\begin{equation}\label{mittag}
dd^c \log|f_{k+1}|^2
\wedge dd^cb_I 
\wedge [V_I'] 
=[D_{k+1}] \wedge dd^c b_I \wedge [V_I'] 
= dd^c b_I \wedge [V]. 
\end{equation}
This proves (ii). 
\end{proof}

\begin{prop}\label{samma} 
Let $u_1,\ldots, u_m$ be psh functions with analytic singularities with
corresponding unbounded loci $Z_1,\ldots, Z_m$. 
Assume that 
\begin{equation}\label{snubbe}
\codim(Z_{i_1}\cap \cdots \cap Z_{i_\ell})\geq \ell,
\end{equation}
for each choice of $1\leq i_1 <\cdots < i_\ell\leq m$. 
Then
\begin{equation}\label{vett}
dd^c u_m\1_{X\setminus Z_m} \wedge \cdots \wedge dd^c
u_1 \1_{X\setminus Z_1}= 
dd^c u_m \wedge \cdots \wedge dd^c u_1,
\end{equation}
where the right hand side is defined in the sense of
Bedford-Taylor-Demailly. 
\end{prop}

\begin{proof}
    The statement is local, so it is enough to assume that $X$ is a fixed relatively compact coordinate neighborhood of any given point.
    We let $u^N_m := \max(u_m,-N)$, which is psh and decreases pointwise to $u_m$ when $N \to \infty$.
    Since $u_m$ has analytic singularities, $u_m^N \equiv -N$ in some neighborhood of $Z_m$.
    We let $u^{N,\varepsilon}_m$ be obtained from $u^N_m$ through convolution with an approximate identity, so that
    $u^{N,\varepsilon}_m$ is smooth, psh and decreases pointwise to $u^N_m$ when $\varepsilon \to 0$. Since $u^N_m \equiv -N$ in a neighborhood
    of $Z_m$, it follows that $u^{N,\varepsilon}_m \equiv -N$ in some smaller neighborhood of $Z_m$ when $\varepsilon$ is small enough.
    We can thus find a sequence $u^{(\iota)}_m$ of smooth psh functions decreasing pointwise to $u_m$ such that 
    each $u^{(\iota)}_m$ is constant in some neighborhood of $Z_m$.

    We then get that
    \begin{multline*}
        dd^c u_m\1_{X\setminus Z_m} \wedge (dd^c u_{m-1} \wedge \cdots \wedge dd^c u_1)=
        \lim_\iota dd^c u_m^{(\iota)} \wedge \1_{X\setminus Z_m} (dd^c u_{m-1} \wedge \cdots \wedge dd^c u_1)= \\
        \lim_\iota dd^c u_m^{(\iota)} \wedge dd^c u_{m-1} \wedge \cdots \wedge dd^c u_1=
        dd^c u_m \wedge dd^c u_{m-1} \wedge \cdots \wedge dd^c u_1,
    \end{multline*}
    where the first equality follows from the definition of the first current, the second equality
    follows since $u^{(\iota)}_m$ is smooth and constant in a neighborhood of $Z_m$,
    and the last equality follows from Proposition \ref{addis} because of \eqref{snubbe}.
    The proposition then follows by induction over $m$.
\end{proof}

Recall that a function $q$ is \emph{quasiplurisubharmonic (qpsh)} if
it is of the form $q=u+a$, where $u$ is psh and $a$ is smooth. We say
that the qpsh function $q$ has \emph{analytic singularities} if $u$ has
analytic singularities. 
The \emph{unbounded locus}
of $q$ is defined as the unbounded locus of $u$.

\begin{lma}\label{paron}
Let $T$ be a current with analytic singularities, let 
$q$ be a qpsh function with analytic singularities with unbounded
locus 
$Z$, and let $U\subset X\setminus Z$ be a constructible set. 
Then,
\[
q  \1_U T := u  \1_U T+ a \1_U T
\]
is independent of the decomposition $q=u+a$, where $u$ is psh and $a$
is smooth. 
\end{lma}

\begin{proof} Let $q=u_1+a_1=u_2+a_2$ be two decompositions of $q$
  such that $u_j$ are psh and $a_j$ are smooth.
The statement is local. Therefore we may approximate
$q$ by convoluting with a sequence of regularizing kernels
$\rho^{(\iota)}$ so that for $j=1,2$,
$u_j^{(\iota)}:=u_j\ast\rho^{(\iota)}$ is a sequence of smooth psh
functions decreasing to $u_j$, and
$a_j^{(\iota)}:=a_j\ast\rho^{(\iota)}$ is a sequence of smooth
functions converging to $a_j$, and for each $\iota$,
$u^{(\iota)}_1+a^{(\iota)}_1=u^{(\iota)}_2+a^{(\iota)}_2$.  
Thus in light of Proposition ~\ref{liseberg} we get
\begin{equation*}
u_1\1_U  T+a_1 \1_U T=
\lim_{\iota} \big (u^{(\iota)}_1 \1_U T + a^{(\iota)}_1 \1_U T \big )
=
\lim_{\iota}\big (u^{(\iota)}_2 \1_U T + a^{(\iota)}_2 \1_U T\big )
=
u_2 \1_U T+a_2 \1_U T.
\end{equation*}
\end{proof}

Let $u$ be a psh function with analytic singularities with unbounded
locus $Z$, let $\alpha$ be a closed smooth $(1,1)$-form, and 
let $T$ be a current (locally) of the form 
\begin{equation}\label{hjartat}
T=\sum \beta_i \wedge T_i,
\end{equation}
where the sum is finite, $\beta_i$ are closed smooth forms, and $T_i$ are currents with analytic
singularities. 
We define the operator
$T\mapsto [dd^c u]_\alpha\wedge T$ by 
\begin{equation*}
[dd^c u]_\alpha\wedge T := dd^c u\wedge \1_{X\setminus Z} T +\alpha
\wedge \1_Z T
:= dd^c (u \1_{X\setminus Z} T) +\alpha \wedge \1_Z T.
\end{equation*}
By Remark ~\ref{stress} this is a well-defined 
current of the form \eqref{hjartat}. 
Using that $\1_{X\setminus Z}(\alpha^\ell\wedge \1_Z (dd^c u)^k)=0$ for $k,\ell\geq 0$, we get that 
\begin{equation}\label{bullen}
[dd^c u]^m_\alpha=[dd^c u]_\alpha\wedge [dd^c u]^{m-1}_\alpha, 
\end{equation}
where $[dd^c u]^m_\alpha$ is defined by \eqref{geten}. 

Next, for currents $T$ of the form \eqref{hjartat}
we define operators $T\mapsto [dd^c u]_\alpha^m\wedge T$
recursively by  
\begin{equation}\label{triangel}
 [dd^c u]_\alpha^k\wedge T:= [dd^c u]_\alpha\wedge [dd^c u]_\alpha^{k-1}\wedge T.
\end{equation}
Again by Remark ~\ref{stress} these are currents of the form
\eqref{hjartat}. In particular, if $u_1,\ldots, u_t$ are psh functions with analytic
singularities, and $\alpha_1,\ldots, \alpha_t$ are closed smooth 
$(1,1)$-forms, the current 
\begin{equation*}
[dd^c u_t]^{m_t}_{\alpha_t}\wedge \cdots \wedge [dd^c u_1]^{m_1}_{\alpha_1}
\end{equation*}
is a well-defined current of the form \eqref{hjartat}. 

Note that if $\beta$ is a closed smooth form, then 
\begin{equation*}
 [dd^c u]_\alpha^k\wedge \beta\wedge T = 
\beta \wedge [dd^c u]_\alpha^k\wedge T. 
\end{equation*}
Indeed, multiplication with $\1_U$ and $dd^c u$ commutes with multiplication with
closed smooth forms.

\begin{remark}\label{massformeln} 
Recall that if $(X,\omega)$ is a 
K\"ahler manifold, then a function 
$\phi\colon X\rightarrow\R\cup\{-\infty\}$ is called {\it
$\omega$-plurisubharmonic ($\omega$-psh)} if locally the function $g+\phi$
is psh, where $g$ is a local potential for $\omega$, i.e.,  
$\omega=dd^cg$. Moreover $\phi$ is said to have {\it analytic singularities} if 
the functions $g+\phi$ have analytic singularities. 
If $\phi$ is an $\omega$-psh function with analytic singularities, 
we can define a global positive current $(\omega + dd^c \phi)^k$, by
locally defining it as $(dd^c (g+\phi))^k$, see
\cite[Lemma~5.1]{ABW}. 
Analogously to \eqref{geten} we can define 
\begin{equation*}
[\omega + dd^c \phi]^m_\omega:=(\omega + dd^c \phi)^m+\sum_{k=0}^{m-1}
\omega^{m-k}\wedge \1_Z (\omega + dd^c\phi)^k, 
\end{equation*}
where $Z$ is the unbounded locus of $\phi$, cf.\ \cite{ABW}. With this
notation, Theorem~1.2 in \cite{ABW} can be formulated as:

\noindent
\it{Let $\phi$ be an $\omega$-psh function with analytic
  singularities on a compact K\"ahler manifold
  $(X,\omega)$ of dimension $n$. Then 
\begin{equation*}\label{north}
\int_X[\omega+dd^c\phi]_\omega^n
      =\int_X\omega^n.
\end{equation*}
}
\end{remark}

\section{Hermitian metrics with analytic singularities on line bundles}\label{linjen}

Let $L\to X$ be a holomorphic line bundle. A \emph{singular hermitian metric} on
$L$, as introduced by Demailly \cite{Dem2},
consists of (possibly infinite) seminorms $\|\cdot\|_{h(x)}$ on $L_x$ for all $x\in X$
such that if $\vartheta: L|_{\mathcal U} \to \mathcal U\times \C $
is a local trivialization of $L$ and $\xi$ is a local section, then 
$\|\xi\|^2_h=|\vartheta (\xi)|^2e^{-\varphi}$, where $\varphi$ is a locally
integrable function in $\U$ called the \emph{(local) weight} of $h$
with respect to $\vartheta$, see, e.g., \cite[Chapter~9.4.D]{Laz2}. 
The metric $h$ is often denoted by $e^{-\varphi}$ or just $\varphi$. 
If $\vartheta': L|_{\mathcal U'} \to \mathcal U'\times \C$ is another
local trivialization, with transition function $g$, then in
$\U\cap\U'$ the corresponding weight $\varphi'$ satisfies
\begin{equation}\label{hostelihost}
\varphi'=\varphi+\log |g|^2.
\end{equation}
From \eqref{hostelihost} and the local integrability of the weights it follows that
the curvature current 
$\Theta_h=\partial\dbar \varphi$ is a well-defined global current on
$X$. The \emph{Chern form} 
$c_1(L,h)=(i/2\pi)\Theta_h =dd^c\varphi$ represents the Chern
class $c_1(L)$.

The metric $e^{-\varphi}$ is \emph{positive} if the weights $\varphi$ are
psh. We say that a positive singular metric $e^{-\varphi}$ has \emph{analytic
  singularities} if the weights $\varphi$ have analytic singularities. 
In view of \eqref{hostelihost} there is a well-defined associated
\emph{unbounded locus} $Z\subset X$, that is a subvariety of $X$, locally defined as the unbounded
loci 
of the $\varphi$. 
Since the local weights are integrable it follows that $Z$ has positive codimension in $X$.

\begin{ex}\label{fagel}
Assume that $s_1,\ldots, s_N$ are nontrivial holomorphic sections of a line bundle 
$L\to X$. Then $h=e^{-\varphi}$ with 
\[
\varphi=\log \sum_{j=1}^N |s_j|^2
\]
is a positive metric with analytic singularities, cf.\
\cite[Example~2.4]{Dem2}. 
In other words, if $\vartheta: L|_{\mathcal U} \to \mathcal U\times
\C$ is a local trivialization and $\xi$ is a local section, then 
\[
\|\xi\|^2_h=\frac{|\vartheta(\xi)|^2}{\sum_{j=1}^N|\vartheta(s_j)|^2}.
\]
\end{ex}

\begin{lma}\label{filmfestival}
Assume that $\varphi_1,\ldots, \varphi_t$ are positive metrics with
analytic singularities on line bundles 
$L_1,\ldots, L_t$ over $X$ with unbounded loci $Z_1,\ldots,
Z_t$, respectively. 
If $U_1,\ldots, U_t$ are
constructible sets contained in $X\setminus Z_1,\ldots, X\setminus
Z_t$, respectively, and $\theta_1,\ldots, \theta_t$ are closed 
$(1,1)$-forms, then the a priori locally defined currents 
\begin{equation}\label{matta}
dd^c\varphi_t\1_{U_t}\wedge\cdots\wedge dd^c\varphi_1\1_{U_1}, 
\end{equation}
\begin{equation}\label{satta}
[dd^c\varphi_t]^{m_t}_{\theta_t}\wedge \cdots \wedge [dd^c\varphi_1]^{m_1}_{\theta_1}
\end{equation}
are globally defined currents;
\eqref{matta} has analytic singularities and \eqref{satta} is of the
form \eqref{hjartat}. 
\end{lma}

\begin{proof}
Since the local weights of $\varphi_j$ are psh functions with analytic
singularities, \eqref{matta} and \eqref{satta} are locally well-defined
and of the desired form in view of Section \ref{general}. Since two
local weights differ by a pluriharmonic function, cf.\
\eqref{hostelihost}, it follows using Lemma \ref{paron} that
they are globally defined. 
\end{proof}

If $\varphi_j=\varphi$, with singular set $Z$, and $U_j=X\setminus
Z$ for all $j$, we write $(dd^c \varphi)^t$ 
for the generalized Monge-Amp\`ere product \eqref{matta}, cf.\ \eqref{defen}.

For the proofs of Theorems \ref{thmA} and \ref{thmB} we will need the
following results.

\begin{prop}\label{dova}
Let $\varphi_1,\ldots, \varphi_t$ be hermitian metrics with analytic
singularities on holomorphic line bundles $L_1,\ldots, L_t$, respectively, over a
complex manifold $X$. Moreover, let $\theta_1,\ldots, \theta_t$ be
first Chern forms of smooth metrics on
$L_1,\ldots, L_t$, respectively. Then 
\begin{equation}\label{rova}
[dd^c\varphi_t]^{m_t}_{\theta_t}\wedge\cdots\wedge
[dd^c\varphi_1]^{m_1}_{\theta_1} 
=
\theta_t^{m_t}\wedge\cdots\wedge \theta_1^{m_1} + dd^c S,
\end{equation}
where $S$ is a current on $X$. 
\end{prop}

\begin{proof}
In view of \eqref{triangel}
{it is enough to prove the result for $m_j=1$,} $j=1,\ldots, t$.
Assume that $\theta_j$ is the first Chern form of the smooth metric $\psi_j$. 
Then note that $\varphi_j-\psi_j$ defines a global
qpsh function on $X$ for each $j$, cf.\ \eqref{hostelihost}.

Let $T$ be a current of the form \eqref{hjartat}. Then 
\begin{multline}\label{snut}
[dd^c\varphi_j]_{\theta_j}\wedge T 
=
dd^c(\varphi_j\1_{X\setminus Z_j} T)+\theta_j\wedge \1_{Z_j} T
=\\
dd^c(\varphi_j\1_{X\setminus Z_j} T) - \theta_j\wedge \1_{X\setminus
  Z_j} T + \theta_j\wedge T
=
dd^c\big ( (\varphi_j- \psi_j) \1_{X\setminus Z_j}
T\big )  + \theta_j\wedge T, 
\end{multline}
where we have used that $dd^c \psi_j=\theta_j$, Lemma ~\ref{paron}, and that  $\1_{X\setminus Z_j}
T$ is closed for the last equality.

Assume that $t=1$. Then it follows from \eqref{snut} (applied to $T=1$
and $j=1$) 
that \eqref{rova} holds with $S=(\varphi_1-\psi_1)\1_{X\setminus
  Z_1}$. In fact, $S=\varphi_1-\psi_1$ since $Z_1$ has positive
codimension in $X$. 
Next assume that \eqref{rova} holds for $t=\kappa$. Then \eqref{snut}
(applied to $T=[dd^c\varphi_\kappa]_{\theta_\kappa}\wedge \cdots\wedge
[dd^c\varphi_1]_{\theta_1}$ and $j=\kappa+1$) gives
\begin{multline*}
[dd^c\varphi_{\kappa+1}]_{\theta_{\kappa+1}}\wedge \cdots\wedge
[dd^c\varphi_1]_{\theta_1} =
dd^c U
+ 
\theta_{\kappa+1} \wedge
[dd^c\varphi_{\kappa}]_{\theta_{\kappa}}\wedge \cdots\wedge
[dd^c\varphi_{1}]_{\theta_{\1}}
=\\
dd^c U+ 
\theta_{\kappa+1}\wedge \big (\theta_{\kappa} \wedge\cdots\wedge \theta_{1}  
+ dd^c S \big )
=
\theta_{\kappa+1}\wedge\cdots\wedge \theta_1 + dd^c
(U+ \theta_{\kappa+1}\wedge S), 
\end{multline*}
where 
\[
U=(\varphi_{\kappa+1}- \psi_{\kappa+1}) \1_{X\setminus Z_j}
[dd^c\varphi_{\kappa}]_{\theta_{\kappa}}\wedge \cdots\wedge
[dd^c\varphi_{1}]_{\theta_{\1}}. 
\]
Thus \eqref{rova} holds for $t=\kappa+1$ and the lemma follows by
induction. 
\end{proof}

\begin{lma}\label{medusa}
Let $\varphi$ be a positive metric with analytic
singularities on a holomorphic line bundle $L$ over a
complex manifold $X$, and let $\theta$ be the first Chern form of a smooth
metric on $L$. 
Let $\varphi_\varepsilon$ be a sequence of
smooth positive metrics decreasing to $\varphi$ 
and let $\omega_\epsilon$ be the corresponding first Chern
forms. Moreover, let $T$ be a current of the form \eqref{hjartat}, 
and let $\beta$ be a test form such that the support of
$dd^c\beta$ does not intersect the unbounded locus $Z$ of $\varphi$. 
Then 
\begin{equation}\label{skjuva}
\int_X [dd^c\varphi]^{m}_{\theta}\wedge T \wedge \beta
=
\lim_{\varepsilon\rightarrow 0}
\int_X \omega_\varepsilon^{m}\wedge 
T \wedge \beta.
\end{equation}
\end{lma}

\begin{proof}
Assume that $\theta$ is the first Chern form of the smooth metric
$\psi$. First note that for any current $S$ of the form \eqref{hjartat} 
\begin{multline}\label{diskho}
[dd^c\varphi]_{\theta}\wedge S - \omega_\eps\wedge S
=
dd^c(\varphi\1_{X\setminus Z} S)+\theta\wedge \1_{Z} S - \omega_\eps\wedge S
=\\
dd^c\big ((\varphi-\varphi_\eps) \1_{X\setminus Z} S\big ) 
+ 
dd^c\big ((\psi-\varphi_\eps) \1_{Z} S\big ). 
\end{multline}
Since $\varphi-\varphi_\eps$ and  $\psi-\varphi_\eps$ are
globally defined qpsh functions, cf.\ \eqref{hostelihost}, in view of Lemma ~\ref{paron}
the currents on the last line of \eqref{diskho} are globally defined
currents of the form \eqref{hjartat}. 

By applying \eqref{diskho} to
$S=[dd^c\varphi]_\theta^{\ell-1}\wedge T$ for $\ell=1,\ldots, m$, it follows that 
\begin{multline}\label{masta} 
[dd^c\varphi]^m_{\theta}\wedge T - \omega^m_\eps\wedge T
=
\sum_{\ell=1}^m \omega_\eps^{m-\ell} \wedge 
\big ([dd^c\varphi]^\ell_{\theta}\wedge T - \omega_\eps \wedge
[dd^c\varphi]^{\ell-1}_{\theta}\wedge T \big )
=\\
\sum_{\ell=1}^m \omega_\eps^{m-\ell} \wedge 
dd^c\big ((\varphi-\varphi_\eps) \1_{X\setminus Z}
[dd^c\varphi]^{\ell-1}_{\theta} \wedge T\big ) 
+ 
\sum_{\ell=1}^m \omega_\eps^{m-\ell} \wedge dd^c \big ( 
(\psi-\varphi_\eps) \1_{Z} [dd^c\varphi]^{\ell-1}_{\theta}
\wedge T\big ).  
\end{multline}
Again, the currents in the last line are well-defined global currents
by Lemma ~\ref{paron}.

Let us integrate one of the currents in the second sum against
$\beta$. Then by Stokes' theorem
\begin{multline*}
\int_X 
\omega_\eps^{m-\ell} \wedge dd^c \big ( 
(\psi-\varphi_\eps) \1_{Z} [dd^c\varphi]^{\ell-1}_{\theta}
\wedge T\big ) \wedge \beta
=\\
\int_X 
\omega_\eps^{m-\ell} \wedge 
(\psi-\varphi_\eps) \1_{Z} [dd^c\varphi]^{\ell-1}_{\theta}
\wedge T  \wedge dd^c \beta=0,  
\end{multline*} 
where the last equality follows since  $\supp (\1_{Z}
[dd^c\varphi]^{\ell-1}_{\theta} \wedge T)\subset Z$ 
is disjoint from  $\supp dd^c \beta$. 
Outside $Z$, the current 
\begin{equation*}\label{onsdag}
\omega_\eps^{m-\ell} \wedge 
(\varphi-\varphi_\eps) \1_{X\setminus Z}
[dd^c\varphi]^{\ell-1}_{\theta} \wedge T
=
(\varphi-\varphi_\eps) (dd^c\varphi_\eps)^{m-\ell}\wedge 
\1_{X\setminus Z} [dd^c\varphi]^{\ell-1}_{\theta} \wedge T
\end{equation*}
converges weakly to $0$ by Proposition ~\ref{addis}. 
Thus, integrating one of the terms in the first sum in the last line
of \eqref{masta} against $\beta$ and taking the limit gives 
\begin{multline*}
\lim_{\varepsilon\rightarrow 0} \int_X 
\omega_\eps^{m-\ell} \wedge dd^c \big ( 
(\varphi-\varphi_\eps) \1_{X\setminus Z} [dd^c\varphi]^{\ell-1}_{\theta}
\wedge T\big ) \wedge \beta
=\\
\lim_{\varepsilon\rightarrow 0}\int_X 
\omega_\eps^{m-\ell} \wedge 
(\varphi-\varphi_\eps) \1_{X\setminus Z} [dd^c\varphi]^{\ell-1}_{\theta}
\wedge T  \wedge dd^c \beta=0. 
\end{multline*}
Now \eqref{skjuva} follows by integrating \eqref{masta} against $\beta$ and
taking the limit. 
\end{proof}

\section{Hermitian metrics with analytic singularities on 
  vector bundles}\label{jensen}

Let $E\to X$ be a holomorphic vector bundle over a complex manifold $X$. 
A \emph{singular hermitian metric} $h$ on $E$ in the sense of
Berndtsson-P\u aun, \cite{BP}, is a measurable function from $X$ to the
space of nonnegative hermitian forms on the fibers. The hermitian
forms are allowed to take the value $\infty$ at some points in the
base (i.e., the norm function $\|\xi\|_h$ is a measurable function with
values in $[0,\infty]$), but for any fiber $E_x$ the subset
$E_0:=\{\xi\in E_x\ ;\ \|\xi\|_{h(x)}<\infty\}$ has to be a linear
subspace, and the restriction of the metric to this subspace must be
an ordinary hermitian form.

Every singular hermitian metric $h$ on $E$
induces a canonical dual singular hermitian metric $h^*$ on $E^*$ such
that $(h^{*})^*=h$ under the natural isomorphism $(E^{*})^*\cong E$,
see, e.g., \cite[Lemma~3.1]{LRRS}. 
Following \cite[Section~3]{BP} we say that $h$ is \emph{Griffiths
  negative} if the function\footnote{The function $\chi_h$ is sometimes called
  the \emph{logarithmic indicatrix} of the (Finsler) metric $h$, see,
  e.g., \cite{Dem99}.}  
\[
\chi_h(x,\xi):=\log \|\xi\|^2_{h(x)}
\] 
is psh on the total space of $E$. Moreover we say that $h$ is \emph{Griffiths positive} if the dual metric
$h^*$ on $E^*$ is negative.

\begin{prop}\label{blabar}
Let $h$ be a singular hermitian metric on a holomorphic vector bundle $E$. Let $0_E$ denote the zero section of $E$.  Then
the following conditions are equivalent. 
\begin{enumerate} 
\item\label{third}
$h$ is Griffiths negative, i.e., $\chi_h$ is psh on the total space of $E$,
\item\label{fourth}
$\chi_h$ is psh on $E\setminus 0_E$, 
\item\label{first}
the function $x\mapsto \log \|u(x)\|^2_{h(x)}$ is psh for each local
section $u$ of $E$,
\item\label{second}
$\log\|u\|^2_{h}$ is psh for each local
nonvanishing section $u$ of $E$.
\end{enumerate}
\end{prop}

\begin{proof}
We first prove that \eqref{first} is equivalent to \eqref{third} and
that \eqref{second} is equivalent to \eqref{fourth}. 
Note that if $u$ is a local holomorphic section of $E$, then $\log
\|u(x)\|^2_{h(x)}=\chi_h\circ u(x)$. Thus if $\chi_h$ is psh, then so is
$\log \|u\|^2_h$ since $u$ is a holomorphic map. Hence \eqref{third} implies \eqref{first}. 
If $u\neq 0$, then it is enough that $\chi_h$ is psh on
$E\setminus0_E$ in order for $u$ to be psh. Thus \eqref{fourth}
implies \eqref{second}. 

For the converses, since plurisubharmonicity is a local property, 
we may assume that $X$ is an open subset of 
$\C^n$ and that $E=X \times \C^r$. To prove
that $\chi_h$ is psh it is then sufficient to prove that $\chi_h \circ\gamma (t)$ is
subharmonic on (the restriction to $E$ of) any complex line $\gamma(t)$ in $\C^n \times \C^r$. We
let $\gamma_0$ and $\gamma_1$ denote the components of $\gamma$ in
$\C^n$ and $\C^r$, respectively. 
If $\gamma_0(t)$ is constant then
\[
\chi_h\circ\gamma(t)=\log \|\gamma_1(t)\|^2_{h(\gamma_0(t))} = \log
\|\gamma_1(t)\|^2_{h_0}, 
\]
where $h_0$ is the constant metric
$h(\gamma_0)$. If $\|u\|^2_h$ is psh for all $u$, then 
$h_0$ has to be finite, and thus since $\gamma_1(t)$ is a holomorphic curve, 
it follows that $\chi_h\circ\gamma$ is subharmonic. 
If $\gamma_0(t)$ is not constant, then note that 
$\gamma_1(t)=u\circ \gamma_0(t)$ for some
(linear) holomorphic function $u$, that can be extended to a
holomorphic section on $X$.  
Thus 
\[
\chi_h\circ\gamma(t)=\log \|\gamma_1(t)\|^2_{h(\gamma_0(t))} = 
\log\|u\circ \gamma_0(t)\|^2_{h(\gamma_0(t))}=\log\|u\|^2_{h}\circ \gamma_0(t).
\]
Since $\gamma_0$ is holomorphic, $\chi_h\circ\gamma$ is subharmonic if
$\|u\|^2_h$ is psh. Hence \eqref{first} implies \eqref{third}. 
To show that $\chi_h$ is psh outside the zero section of $E$, $u$ can be
chosen nonvanishing, and thus \eqref{fourth} follows from
\eqref{second}.

\smallskip 

Next we prove that \eqref{third} is equivalent to
\eqref{fourth}. Clearly \eqref{third} implies \eqref{fourth}. 
To prove the converse assume that $\chi_h$ is psh on $E\setminus
0_E$. Then $\|\xi\|^2_{h(x)}$ is finite on $E\setminus0_E$ and thus
by homogeneity it must vanish on $0_E$, which means that
$\chi_h|_{0_E}\equiv -\infty$. It follows that $\chi_h$ trivially satisfies the
sub-mean value property at each point of $0_E$. 

To prove that $\chi_h$ is upper semicontinuous at $0_E$, choose
$(x_0,\xi_0)\in0_E$ and let $(x_k,\xi_k)$ be a sequence of points
converging to $(x_0,\xi_0)$. We need to prove that $\lim_k \chi_h
(x_k,\xi_k)=-\infty$. As above, we may assume that 
$X$ is an open subset of $\C^n_x$ and $E=X\times \C^r_\xi$, 
and moreover that 
 $0_E=\{\xi=0\}$ and $(x_0,\xi_0)=(0,0)$. 
Also we may assume that $(x_k,\xi_k)$ are contained in the set 
$\{|x|\leq 1, |\xi|\leq 1\}$. Let $C$ be the compact ``cylinder'' 
\[
C=\{|x|\leq 1, |\xi|=1\}.
\]
Since $\chi_h$ is psh and thus upper semi-continuous outside $0_E$,
$\chi_h|_C\leq M$ for some $M<\infty$. By homogeneity it follows that 
\[
\chi_h(x_k,\xi_k)\leq M+\log |\xi_k|^2\to -\infty.
\]
Thus $\chi_h(x,\xi)$ is upper semicontinuous at $0_E$ and hence it is
psh in $E$. 
\end{proof}

Let $\pi:\P(E)\to X$ denote the \emph{projectivization} of $E$, i.e.,
the projective bundle of lines in the dual bundle $E^*$ of $E$, i.e.,
$\P(E)_x=\P(E_x^*)$. 
The pullback bundle $\pi^*E^*\to \P(E)$ then carries a tautological
line bundle 
\[
\Ok_{\P(E)}(-1)=\{(x,[\xi];v), v\in \C\xi\} \subset \pi^*E^*. 
\]
Let $e^\varphi$ denote the restriction of $\pi^*h^*$ to
$\Ok_{\P(E)}(-1)$. Then $e^{-\varphi}$ is the dual metric on the dual
line bundle $\Ok_{\P(E)}(1)$. If $E$ is a line bundle,
then $\Ok_{\P(E)}(1)\cong E$ and $e^{-\varphi}\cong h$.

Let us describe $\varphi$ in a local trivialization. After possibly
shrinking $X$ we may assume that $E= X \times \C^r$; then
$\P(E)=X\times \P^{r-1}$. 
For $i=1,\ldots, r$, let 
\[
\U_i=\{(x,[\xi])\in \P(E), \xi_i\neq 0\}.
\]
Then $\{\U_i\}$ is an open cover of $\P(E)$ and $\Ok_{\P(E)}(-1)$ is defined by the trivializations 
\[
\vartheta_i:\Ok_{\P(E)}(-1)|_{\U_i}\to \U_i\times \C, ~~~~
(x,[\xi];v)\mapsto (x,[\xi];v_i).
\]
Now, on $\U_i$, 
\begin{equation}\label{trottast}
\|v\|^2_{\pi^*h^*(x,[\xi])}=
|\vartheta_i(v)|^2e^{\varphi_i(x,[\xi])}=
|v_i|^2e^{\varphi_i(x,[\xi])}.
\end{equation}
Moreover, since $\pi^*h^*$ is a pullback metric
\begin{equation}\label{trottare}
\|v\|_{\pi^*h^*(x,[\xi])}=
\|v\|_{h^*(x)}.
\end{equation}
By applying \eqref{trottast} and \eqref{trottare} to $v=\xi$ we get 
\begin{equation}\label{dummemosse}
\varphi_i(x,[\xi])=
\log \|\xi/\xi_i\|^2_{h^*(x)} =\chi_{h^*}(x,\xi/\xi_i). 
\end{equation}
Note that this is well-defined since the second and third expressions
only depend on $[\xi]$.

\begin{prop}\label{storemosse}
Let $h$ be a singular hermitian metric on a holomorphic vector bundle $E$. Then
$h$ is Griffiths positive if and only if the induced singular metric
$e^{-\varphi}$ on $\Ok_{\P(E)}(1)$ is positive. 
\end{prop}

\begin{proof}
Since this is a local statement we may assume that we are in the
situation above. 
Then $e^{-\varphi}$ is positive if and only if $\varphi_i$ is psh on
$\U_i$ for all $i$. Moreover, by Proposition ~\ref{blabar}, $h$ is
Griffiths positive if and only if $\chi_{h^*}(x,\xi)$ is psh on $E$ or
equivalently on $E\setminus 0_E$. 

Since $\xi_i\neq 0$ on $\U_i$, in view of \eqref{dummemosse},
$\varphi_i$ is psh there if $\chi_{h^*}$ is psh. Thus $e^{-\varphi}$
is positive if $h$ is Griffiths positive. 
For the converse, if $(x,\xi)\in E\setminus 0_E$, then $\xi_i\neq 0$ for
some $i$ in some neighborhood $\U$ of $(x,\xi)$. Then, by
\eqref{dummemosse}, 
\begin{equation}\label{jul}
\chi_{h^*}(x,\xi)=\chi_{h^*}(x,\xi/\xi_i)+\log |\xi_i|^2=\varphi_i(x,[\xi])+\log |\xi_i|^2
\end{equation}
there. 
Since $\xi\to[\xi]$ is holomorphic and $\log |\xi_i|^2$ is
pluriharmonic where $\xi_i\neq 0$, it follows that
$\chi_{h^*}$ is psh in $\U$ if $\varphi_i$ is psh. We conclude
that $h$
is Griffiths positive if $e^{-\varphi}$ is positive. 
\end{proof}

\begin{df}\label{fjadrar}
We say that a Griffiths positive hermitian metric has \emph{analytic
  singularities} if the induced positive metric $e^{-\varphi}$ on $\Ok_{\P(E)}(1)$ has analytic
singularities.
\end{df}

\begin{prop}\label{gummistovlar}
Let $h$ be a Griffiths positive hermitian metric on a holomorphic vector bundle
$E$. Then $h$ has analytic singularities if and only if $\chi_{h^*}$ is
psh with analytic singularities on $E\setminus 0_E$. 
\end{prop}

\begin{proof}
Let us assume that we are in the local situation above. 
Then $h$ has analytic singularities if and only if $\varphi_i$ are
psh with analytic singularities for all $i$. In view of the proof of
Proposition ~\ref{storemosse} this is in turn equivalent to that $\chi_{h^*}$ is
psh with analytic singularities on $E\setminus 0_E$.
\end{proof}

We do not
know whether it is possible to express analytic singularities of $h$
in terms (of analytic singularities) of the functions $\log \|u\|^2_{h}$.

\begin{ex}\label{hosono}
In \cite[Example~3.6]{Hos} Hosono constructed a family of examples of singular
hermitian metrics that generalize the metrics in Example ~\ref{fagel}: 
Assume that 
$E\to X$ is a holomorphic vector bundle with global
holomorphic sections $s_1,\ldots, s_N$. 
Let $s$ be the morphism from the dual bundle $E^*$ to the trivial
bundle $X\times\C^N$ given by $(x,\xi) \mapsto (s_1(x,\xi),\ldots,
s_N(x,\xi))$ and let $h^*$ be the pullback under $s$ of the trivial
metric on $X\times \C^N$, i.e., 
\begin{equation*}
\la \xi,\eta \ra_{h^*(x)} := \la s(x,\xi), s(x,\eta) \ra.
\end{equation*}
Then 
\[
\|\xi\|^2_{h^*(x)}=|s(x,\xi)|^2=\sum_j |s_j(x,\xi)|^2.
\]
It follows that $\chi_{h^*}(\xi,x)=\log |s(x,\xi)|^2$ is psh with analytic
singularities on $E^*$. Thus by Proposition ~\ref{gummistovlar} the
dual metric $h$ of $h^*$ on $E=(E^*)^*$ is
Griffiths positive with analytic singularities. 
\end{ex}

Given a Griffiths positive singular metric $h$, $\log \det h^*$ is
psh, see \cite[Proposition~1.3]{R}, and we can define the
\emph{degeneracy locus} of $h$ as the unbounded locus of $\log \det
h^*$. The following lemma gives alternative definitions in terms of (the unbounded loci)
of  $\chi_{h^*}$ and $\varphi$.

\begin{lma}\label{annandag}
Assume that $h$ is a Griffiths positive singular metric on $E\to X$. Then, using the
notation from above and denoting the projection $E^*\to X$ by $p$,  
\begin{equation}\label{julafton}
L(\log \det h^*)=p\big (L(\chi_{h^*})\setminus 0_E \big ) = \pi \big
(L(\varphi)\big ).
\end{equation}
\end{lma}

In particular, it follows that if $h$ has analytic singularities, then
the degeneracy locus of $h$ is a subvariety of $X$.

\begin{proof}
In view of \eqref{jul}, $(x,\xi)\in E\setminus 0_E$ is in
$L(\chi_{h^*})$ if and only if $(x,[\xi])\in \P(E)$ is in
$L(\varphi)$, and thus the
  second equality in \eqref{julafton} follows. 
The inclusion $\pi \big (L(\varphi)\big )\subset L(\log \det h^*)$ is
an immediate consequence of Lemma ~3.7 in \cite{LRRS}.

Thus it remains to prove that 
\begin{equation}\label{juldan}
L(\log \det h^*)\subset p\big (L(\chi_{h^*})\setminus 0_E \big ). 
\end{equation}
Since the statement is local we may assume that $E=X\times \C^r$. 
Take $x\in L(\log \det h^*)$. Then, by definition there is a sequence
$x_k\to x$ such that $\log \det h^*(x_k)\to -\infty$. This means that there
is a sequence $\varepsilon_k\to 0$ such that 
$\det h^*(x_k) < \varepsilon_k^r$, which implies that $h^*(x_k)$ has at
least one eigenvalue less than $\varepsilon_k$. Thus there are
$\xi_k\in E_{x_k}^*=\C^r$ such that $\|\xi_k\|_{\C^r}=1$ and
$\|\xi_k\|_{h^*(x_k)}<\varepsilon_k$. 
Since $\|\xi_k\|_{\C^r}=1$, $\{\xi_k\}$ has at least one accumulation
point $\xi$ in $\C^r$ and thus we can
find a subsequence $(x_k,\xi_k)\to (x,\xi)$. 
Since
$\|\xi_k\|_{h^*(x_k)}\to 0$, $(x,\xi)\in L(\chi_{h^*})$.
Moreover, since $\|\xi_k\|_{\C^r}=1$, $\|\xi\|_{\C^r}\neq 0$ and thus $x\in
p\big (L(\chi_{h^*})\setminus 0_E \big )$, which proves
\eqref{juldan}.
\end{proof}

\section{Construction of Segre and Chern currents, proof of Theorem ~\ref{thmA}}\label{konstruera}

\subsection{Construction, basic properties}\label{plommon}
Assume that $X$ is a complex manifold of dimension $n$, that $E\to
X$ is a holomorphic vector bundle of rank $r$, and that 
$h$ is a Griffiths positive hermitian metric with analytic
singularities on $E$. 
Let $\pi:\P(E)\to X$
be the projectivization of $E$ and let $\varphi$ denote the metric on
$L:=\Ok_{\P(E)}(1)\to\P(E)$ induced by $h$.  Then $\varphi$ has analytic
singularities, cf.\ Definition ~\ref{fjadrar}; 
let $Z\subset \P(E)$ denote the unbounded locus of $\varphi$. 
Moreover, assume that $\psi$ is a smooth metric on $L$ and let $\theta$ be the
corresponding first Chern form. 

Next, 
let $E_1,\ldots, E_t$ be $t$ 
disjoint copies of $E$, let $\pi_i$ denote the projections $\P(E_i)\to X$, and
$p_i$ the identifications $\P(E_i)\to \P(E)$. 
Let $\tilde \varphi_i$ denote the metric $p_i^*\varphi$ on $\widetilde
L_i:=p_i^* L
\to \P(E_i)$ induced by $h$ with unbounded locus $\widetilde
Z_i:=p_i^{-1}(Z)$ and let
$\tilde \theta_i=p_i^*\theta$  and $\tilde\psi_i=p_i^*\psi$. 
Moreover, let $Y$ be the fiber product
\[
Y=\P(E_t)\times_X\cdots \times_X \P(E_1), 
\]
with projections $\varpi_i: Y\to \P(E_i)$ and $\pi: Y\to X$. 
Let $\varphi_i$ denote the pullback metric $\varpi_i^*\tilde
\varphi_i$ on $L_i:=\varpi_i^* \widetilde L_i$ with unbounded
locus 
$Z_i:=\varpi_i^{-1}(\widetilde Z_i)$ and let
$\theta_i=\varpi_i^*\tilde\theta_i$ and $\psi_i=\varpi_i^*
\tilde\psi_i$. 
Now, in view of Lemma ~\ref{filmfestival}, \eqref{korsbar} is a
well-defined $(k,k)$-current.

\begin{remark}\label{bibliotek}
Let $V$ be the degeneracy locus of $h$. Then in $Y\setminus \pi^{-1}
V$, $\varphi_j$ are locally bounded by Lemma \ref{annandag},  and thus 
\[
[dd^c\varphi_t]^{k_t+r-1}_{\theta_t}\wedge \cdots \wedge
[dd^c\varphi_1]^{k_1+r-1}_{\theta_1} 
=
(dd^c\varphi_t)^{k_t+r-1}\wedge \cdots \wedge
(dd^c\varphi_1)^{k_1+r-1},
\]
where the right hand side is locally defined in the sense of Bedford-Taylor. 
Hence outside $V$, 
\[
s_{k_t}(E,h,\theta)\wedge\cdots\wedge s_{k_1}(E,h,\theta)
=
(-1)^k\pi_* \big ( 
(dd^c\varphi_t)^{k_t+r-1}\wedge \cdots \wedge
(dd^c\varphi_1)^{k_1+r-1}\big );
\]
in particular it is independent of $\theta$, and thus so are $c_k(E,
h, \theta)$ and $s_k(E,h,\theta)$. 
\end{remark}

\begin{lma}\label{institution}
Let $X$, $E$, and $h$ be as above. 
Given $x\in X$, there is a neighborhood $x\in\U\subset X$ such that
in $\U$,
\begin{equation*}
s_{k_t}(E,h,\theta)\wedge\cdots\wedge s_{k_1}(E,h,\theta) = S_+-S_-,
\end{equation*}
where $S_+, S_-$ are closed positive currents.
\end{lma}

In view of \eqref{koko}, it follows that
$s_k(E,h,\theta)$ and $c_k(E,h,\theta)$ are differences of positive
closed $(k,k)$-currents; 
this proves the first part of Theorem~\ref{thmA}.

\begin{proof} 
Let us use the notation from above. 
Since the statement is local we may assume that $X$ is an open
neighborhood of $x$ in $\C^n$ and that $E=X\times \C^r$ is a trivial bundle. 
Then $\P(E_j)\cong X\times Y_j$, where
$Y_j\cong \P^{r-1}$. 
Let $\rho_j$ be the projection $\P(E_j)\to \P^{r-1}$.  Moreover, let
$\omega_\FS$ denote the Fubini-Study metric on $\P^{r-1}$, let
$\omega_0$ be the standard Euclidean metric on $X$,  
and let $\omega_j=\rho_j^*\omega_\FS+ \pi_j^* \omega_0$. 
Then for some large enough
$C>0$, there is a neighborhood $x\in \U\subset X$ such that $\tilde \alpha_j := C\omega_j$ satisfies that 
\[
\tilde\beta_j:=\tilde\alpha_j+\tilde\theta_j \geq 0
\]
in $\pi_j^{-1}\U\subset \P(E_j)$ for each $j$. 
Let $\alpha_j=\varpi_j^*\tilde\alpha_j$ and
$\beta_j=\varpi_j^*\tilde\beta_j$ be the corresponding closed
$(1,1)$-forms on $\pi^{-1}\U \subset Y$, so that $\theta_j=\beta_j-\alpha_j$. 

We claim that in $\pi^{-1} \U$, for $m_j\geq 1$, 
\[
[dd^c\varphi_t]^{m_t}_{\theta_t}\wedge \cdots \wedge
[dd^c\varphi_1]^{m_1}_{\theta_1} 
=
T_+-T_-,
\]
where $T_\pm$ are closed positive currents with analytic
singularities, cf.\ the beginning of Section ~\ref{general}.
Then in view of \eqref{korsbar}, 
$s_{k_t}(E,h,\theta)\wedge\cdots\wedge s_{k_1}(E,h,\theta)$
is of the desired form in $\U$. 

To prove the claim, first
note that if $m_1=0$, then, 
\[
[dd^c\varphi_1]^{m_1}_{\theta_1}=[dd^c\varphi_1]^0_{\theta_1}=1\geq 0.
\]
Next, assume that 
\[
T:=
[dd^c\varphi_\kappa]^{m_\kappa-1}_{\theta_\kappa}\wedge \cdots \wedge
[dd^c\varphi_1]^{m_1}_{\theta_1} 
=
T_+-T_-,
\]
where $m_\kappa\geq 1$, and where $T_\pm$ are as above. 
Then 
\begin{multline*}
[dd^c\varphi_\kappa]^{m_\kappa}_{\theta_\kappa}\wedge \cdots \wedge
[dd^c\varphi_1]^{m_1}_{\theta_1}
=
[dd^c\varphi_\kappa]_{\theta_\kappa}\wedge T
=
dd^c \varphi_\kappa \wedge \1_{Y\setminus Z_\kappa}T 
+ \theta_\kappa \wedge \1_{Z_\kappa} T
=\\
dd^c \varphi_\kappa \wedge \1_{Y\setminus Z_\kappa}(T_+-T_-) 
+ (\beta_\kappa-\alpha_\kappa) \wedge \1_{Z_\kappa} (T_+ - T_-)
=\\
\big ( dd^c \varphi_\kappa \wedge \1_{Y\setminus Z_\kappa} T_+
+\beta_\kappa \wedge \1_{Z_\kappa} T_+
+\alpha_\kappa \wedge \1_{Z_\kappa}  T_- \big )
-\\
\big ( dd^c \varphi_\kappa \wedge \1_{Y\setminus Z_\kappa}T_-
+\beta_\kappa \wedge \1_{Z_\kappa}  T_-
+ \alpha_\kappa \wedge \1_{Z_\kappa} T_+\big ) =:T'_+-T'_-.
\end{multline*} 
Since $T_\pm$ are closed positive currents with analytic
singularities and $\beta_\kappa$ and $\alpha_\kappa$ are positive
$(1,1)$-forms, $T'_\pm$ are well-defined closed positive currents with
analytic singularities. 
The claim now follows by induction. 
\end{proof}

\subsection{Comparison to the smooth case, proof of statement \eqref{andra}
  in Theorem ~\ref{thmA}}\label{sakura}

Note that to prove statement \eqref{andra} in Theorem ~\ref{thmA} it suffices to
show that 
\begin{equation}\label{gulkork}
s_{k_t}(E,h,\theta)\wedge\cdots\wedge s_{k_1}(E,h,\theta) = 
s_{k_t}(E,h)\wedge\cdots\wedge s_{k_1}(E,h)
\end{equation}
when $h$ is smooth. 

To this end, assume that $h$, and thus $\varphi$, is smooth. 
Let $\alpha_j$ be the smooth form 
\[
\alpha_j:=(dd^c\tilde \varphi_j)^{k_j+r-1},
\]
where we use the notation from Section
~\ref{plommon}. 
Then $s_{k_j}(E,h)=(-1)^{k_j}(\pi_j)_*\alpha_j$, cf.\ \eqref{ny}, and thus 
\begin{equation}\label{rodhake}
s_{k_t}(E,h)\wedge\cdots\wedge s_{k_1}(E,h)
=
(-1)^k (\pi_t)_*\alpha_t\wedge\cdots\wedge (\pi_1)_*\alpha_1.
\end{equation}
Note that in this case
\[
\varpi_j^*\alpha_j=(dd^c\varpi_j^*\tilde\varphi_j)^{k_j+r-1}=
(dd^c \varphi_j)^{k_j+r-1},
\]
and thus in view of Remark \ref{bibliotek}
\[
\varpi_t^*\alpha_t\wedge\cdots\wedge \varpi_1^*\alpha_1=
[dd^c\varphi_t]_{\theta_t}^{k_t+r-1}\wedge\cdots\wedge [dd^c\varphi_1]_{\theta_1}^{k_1+r-1},
\]
so that 
\begin{equation*}
s_{k_t}(E,h,\theta)\wedge\cdots\wedge s_{k_1}(E,h,\theta) 
=
(-1)^k \pi_* \big ( \varpi_t^*\alpha_t\wedge\cdots\wedge \varpi_1^*\alpha_1
\big ). 
\end{equation*}
Now \eqref{gulkork}
follows from the following lemma 
(with $Y_j=\P(E_j)$).

\begin{lma} \label{forwards}
Let $X$ be a complex manifold, let $\pi_j\colon Y_j\rightarrow X$,
$j=1,\ldots, t$, be proper submersions, and let $Y$ be the fiber
product $Y:=Y_t\times_X\cdots \times_X Y_1$ with projections 
$\varpi_j:Y\to Y_j$ and $\pi:Y\to
X$. Let $\alpha_1$ be a 
current on $Y_1$, and let $\alpha_2,\ldots, \alpha_t$ be smooth forms
on $Y_2,\ldots, Y_t$, respectively. 
Then 
\begin{equation}\label{fiskmoll}
\pi_* \big ( \varpi_t^*\alpha_t\wedge\cdots\wedge \varpi_1^*\alpha_1
\big )
=
(\pi_t)_*\alpha_t \wedge \cdots\wedge (\pi_1)_* \alpha_1.
\end{equation}
\end{lma}

\begin{proof}
By induction it is enough to prove the case $t = 2$. It is also enough to prove \eqref{fiskmoll} locally. We may
therefore assume that $Y_j\cong X\times Z_j$, where $Z_j$ is a manifold for $j=1,2$. 
It is readily verified that 
\begin{equation}\label{eq:pushPull2}
\pi_1^\ast(\pi_2)_\ast \alpha_2 = (\varpi_1)_\ast \varpi_2^* \alpha_2
\end{equation}
since the pushforwards on both sides are just integration along
$Z_2$.
By \eqref{daghem}, \eqref{daghem2}, \eqref{eq:pushPull2}, and the fact that $\pi_1 \circ \varpi_1 = \pi$, we obtain that
	\begin{multline*}
	(\pi_2)_\ast \alpha_2 \wedge (\pi_1)_\ast \alpha_1 
	= (\pi_1)_\ast (\pi_1^\ast(\pi_2)_\ast \alpha_2 \wedge \alpha_1) 
	= (\pi_1)_\ast( (\varpi_1)_\ast\varpi_2^\ast \alpha_2 \wedge \alpha_1)\\
	= (\pi_1)_\ast (\varpi_1)_\ast (\varpi_2^\ast \alpha_2 \wedge \varpi_1^\ast \alpha_1)
	=\pi_\ast (\varpi_2^\ast \alpha_2 \wedge \varpi_1^\ast \alpha_1).
	\end{multline*}
\end{proof}

\subsection{The cohomology class, proof of statement \eqref{forsta} in 
  Theorem ~\ref{thmA}}\label{forstabevis}

Note in view of \eqref{koko} that to prove statement \eqref{forsta} in
Theorem ~\ref{thmA}
it is enough to prove that $s_{k_t}(E,h,\theta)\wedge\cdots\wedge
s_{k_1}(E,h,\theta)$ is cohomologous to $s_{k_t}(E,g) \wedge\cdots\wedge
s_{k_1}(E,g)$, where $g$ is a smooth metric on $E$.

From Proposition ~\ref{dova} we have that 
\begin{multline*}
s_{k_t}(E,h,\theta)\wedge\cdots\wedge s_{k_1}(E,h,\theta)
=
(-1)^k\pi_* \big ([dd^c\varphi_t]^{k_t+r-1}_{\theta_t}\wedge\cdots\wedge
[dd^c\varphi_1]^{k_1+r-1}_{\theta_1} \big )
= \\
(-1)^k\pi_* \big ( \theta_t^{k_t+r-1}\wedge\cdots\wedge \theta_1^{k_1+r-1} \big ) +
(-1)^kdd^c \pi_* S,
\end{multline*}
for some current $S$; here we have used the notation from Section ~\ref{plommon}.

By applying Lemma \ref{forwards} to $\tilde \theta_j^{k_j+r-1}$, noting
that $\theta_j^{k_j+r-1}=\varpi_j^*\tilde \theta_j^{k_j+r-1}$, we get that 
\begin{equation}\label{hosta}
\pi_* \big ( \theta_t^{k_t+r-1}\wedge\cdots\wedge \theta_1^{k_1+r-1} \big )
=
(\pi_t)_*\tilde\theta_t^{k_t+r-1}\wedge\cdots\wedge (\pi_1)_*\tilde\theta_1^{k_1+r-1}.
\end{equation}
Let $\eta$ be the metric on $\Ok_{\P(E)}(1)$ associated with $g$. Then
$\theta$ is cohomologous to $dd^c\eta$ and since 
$\pi_*$ commutes
with exterior differentiation, it follows from \eqref{ny} that
$(-1)^{k_j}(\pi_j)_*\tilde\theta_j^{k_j+r-1}$ is a form in the class of
$s_{k_j}(E, g)$.
It follows that $(-1)^k$ times right hand side of \eqref{hosta} is
cohomologous to $s_{k_t}(E,g) \wedge\cdots\wedge s_{k_1}(E,g)$, and
we conclude that so is $s_{k_t}(E,h,\theta)\wedge\cdots\wedge
s_{k_1}(E,h,\theta)$.

\subsection{Lelong numbers, proof of statement \eqref{tredje} in
  Theorem ~\ref{thmA}}\label{lelongsection} 

We begin by recalling the definition of the Lelong number of a closed positive current.
We assume that we are on a complex manifold $X$ of dimension $n$ and that around a given point
$a \in X$, we have local coordinates $z$.
If $T$ is a closed positive $(p,p)$-current on $X$, then the \emph{Lelong number} of $T$ at $a$ can be defined
as
\begin{equation}\label{kalsongdef}
  \nu(T,a) := \int {\mathbf{1}_{\{a\}}}  (dd^c \log
  |z-a|^2)^{n-p}\wedge T,
\end{equation}
which is independent of the local coordinate system,
see for example \cite{Dem}*{Definition~III.5.4 and Corollary III.7.2}.
Since $L(\log |z-a|^2) = \{ a \}$, which has codimension $n$,
the product in the integrand is indeed well-defined.

Note that the definition of Lelong numbers can be extended to currents
that are locally 
of the form $T_+ - T_-$, where $T_\pm$ are closed positive currents,
through $\nu(T,a) := \nu(T_+,a) - \nu(T_-,a)$ if $T=T_+-T_-$ in a
neighborhood of $a$.
In particular, in view of Lemma ~\ref{institution} the Lelong numbers are defined for the currents
$s_{k_t}(E, h,\theta)\wedge\cdots\wedge s_{k_1}(E,h,\theta)$, and thus
in particular for $c_k(E,h,\theta)$.

\begin{remark}\label{lelongkalsong}
Let us consider \eqref{kalsongdef}. For simplicity, assume that
$a=0$. 
Note that by the dimension principle 
for any $(p,p)$-current $T$ that is (locally) the difference of two
closed positive currents, 
\begin{equation*} 
(dd^c \log |z|^2)^{n-p} \wedge T
= dd^c \log |z|^2 \1_{X\setminus
  \{0\}} \wedge \cdots \wedge dd^c \log |z|^2 \mathbf{1}_{X\setminus
  \{0\}} \wedge T,
\end{equation*}
cf.\ Proposition ~\ref{propB}. 

Now assume that $T=s_{k_t}(E,h,\theta)\wedge\cdots\wedge
s_{k_1}(E,h,\theta)$, i.e., $T=(-1)^k \pi_*\mu$, where 
\[
\mu=[dd^c\varphi_t]^{k_t+r-1}_{\theta_t}\wedge \cdots\wedge 
[dd^c\varphi_1]^{k_1+r-1}_{\theta_1}
\]
and we are using the notation from Section ~\ref{plommon}. 
Notice that $\log|\pi^* z|^2$ is psh with analytic singularities on $Y$
with unbounded locus $Z:=\pi^{-1}\{0\}$. 
Thus, in view of Lemma ~\ref{filmfestival}, 
arguing as in  the proof of Lemma ~\ref{institution} 
(regarding $\log|\pi^* z|^2$ as a metric on the trivial line bundle
over $Y$), one gets that 
\[
dd^c \log |\pi^* z|^2 \1_{Y\setminus Z} \wedge \cdots \wedge dd^c \log
|\pi^* z|^2 \1_{Y\setminus Z} \wedge \mu 
\]
is a globally defined current 
that in a neighborhood of $Z$ is the difference of two closed positive
currents with analytic singularities.

Next, note that if $u^{(\iota)}$ is a sequence of smooth psh
functions decreasing to $\log|z|^2$, then $\pi^* u^{(\iota)}$ is a
sequence of smooth psh functions decreasing to $\log |\pi^*z|^2$. 
Using \eqref{daghem}, Propositions ~\ref{addis} and ~\ref{liseberg}, and \eqref{bla} we conclude that 
\[
\1_{\{0\}}(dd^c \log |z|^2)^{n-k} \wedge T 
= 
\pi_* \big (\1_Z dd^c \log |\pi^* z|^2 \1_{Y\setminus Z} \wedge\cdots
\wedge dd^c \log
|\pi^* z|^2\1_{Y\setminus Z} \wedge \mu \big ).
\]
\end{remark}

\begin{proof}[Proof of statement \eqref{tredje} in Theorem \ref{thmA}]
Let us choose coordinates so that $a=0$. Since Lelong numbers are locally defined, cf.\ \eqref{kalsongdef}, we
may assume that we are in a neighborhood $0\in\U\subset X$ as in the
proof of Lemma ~\ref{institution}. 
Let $\theta$ and $\theta'$ be two first Chern forms on
$\Ok_{\P(E)}(1)$ corresponding to smooth metrics $\psi$ and $\psi'$, respectively, and let  
$\mu=[dd^c\varphi_t]^{k_t+r-1}_{\theta_t}\wedge \cdots\wedge
[dd^c\varphi_1]^{k_1+r-1}_{\theta_1}$ and
$\mu'=[dd^c\varphi_t]^{k_t+r-1}_{\theta'_t}\wedge \cdots\wedge
[dd^c\varphi_1]^{k_1+r-1}_{\theta'_1}$ be the corresponding currents
on $Y$, where we use the notation from Section ~\ref{plommon} and
$\theta'_j$ is defined analogously to $\theta_j$. 

We claim that in $\pi^{-1}\U$
\begin{equation} \label{tennis}
\mu-\mu'=\sum d\beta_i \wedge \mu_i,
\end{equation}
where $\beta_i$ are smooth forms and $\mu_i$ are closed positive 
currents with analytic singularities. 
Using the notation from Remark ~\ref{lelongkalsong}, let $S\wedge T$
denote the operator 
$T\mapsto \1_Z dd^c \log
|\pi^* z|^2\1_{Y\setminus Z}\wedge \cdots \wedge dd^c \log
|\pi^* z|^2 \1_{Y\setminus Z} \wedge T$.
Since $\beta_i$ is smooth, applying $S$ commutes with multiplication
with $d\beta_i$, and thus, since $S\wedge \mu_i$ is closed, we get 
\begin{equation*}
    S \wedge d\beta_i \wedge \mu_i = d\beta_i \wedge S \wedge \mu_i =
    d(\beta_i \wedge S \wedge \mu_i)=:d \tau_i, 
\end{equation*}
where $\tau_i =\beta_i \wedge S \wedge \mu_i$ 
has support on $Z$.

Hence, in view of Remark ~\ref{lelongkalsong}, taking the claim for granted, 
\begin{multline*}
(-1)^k \Big (\nu\big (s_{k_t}(E,h,\theta)\wedge \cdots\wedge
s_{k_1}(E,h,\theta), 0\big )
-
\nu\big (s_{k_t}(E,h,\theta')\wedge \cdots\wedge
s_{k_1}(E,h,\theta'), 0\big )\Big )
=\\
\nu(\pi_* \mu,0) - \nu(\pi_* \mu',0) = \int \pi_*( S \wedge (\mu-\mu'))
= \sum_i \int d\pi_* \tau_i = 0,
\end{multline*}
where the last equality follows by Stokes' theorem since the $\pi_*\tau_i$ have support
on $\pi (Z)=\{0\}$. 
This proves \eqref{tredje} in Theorem~\ref{thmA}. 

\smallskip

It remains to prove the claim. 
First note that, since 
$[dd^c\varphi_1]^0_{\theta_1}=
[dd^c\varphi_1]^0_{\theta'_1}=1$,  $[dd^c\varphi_1]^0_{\theta_1}-
[dd^c\varphi_1]^0_{\theta'_1}$
vanishes and is in particular of the form \eqref{tennis}. 
Next assume that we have proven that 
\[T-T':=
[dd^c\varphi_\kappa]^{m_\kappa-1}_{\theta_\kappa}\wedge \cdots \wedge
[dd^c\varphi_1]^{m_1}_{\theta_1} 
-
[dd^c\varphi_\kappa]^{m_\kappa-1}_{\theta'_\kappa}\wedge \cdots \wedge
[dd^c\varphi_1]^{m_1}_{\theta'_1} 
= \sum_i d\gamma_i \wedge T_i
\]
for some smooth forms $\gamma_i$ and closed positive currents with
analytic singularities $T_i$, where $m_\kappa\geq 1$.

By the assumption on $\U$,  
$T=T_+-T_-$ in $\pi^{-1}\U$, where $T_\pm$ are closed positive currents with analytic
singularities. 
Now 
\begin{multline*}
[dd^c\varphi_\kappa]^{m_\kappa}_{\theta_\kappa}\wedge \cdots \wedge
[dd^c\varphi_1]^{m_1}_{\theta_1} 
-
[dd^c\varphi_\kappa]^{m_\kappa}_{\theta'_\kappa}\wedge \cdots \wedge
[dd^c\varphi_1]^{m_1}_{\theta'_1} 
=
[dd^c\varphi_\kappa]_{\theta_\kappa}\wedge T -
[dd^c\varphi_\kappa]_{\theta'_\kappa}\wedge T'
=\\
 dd^c\varphi_\kappa \wedge \1_{Y\setminus Z_\kappa} (T-T') + 
(\theta_\kappa -\theta'_\kappa) \wedge \1_{Z_\kappa} T+ 
\theta'_\kappa \wedge \1_{Z_\kappa} (T-T')
=\\
\sum_i  d\gamma_i \wedge  dd^c\varphi_\kappa \wedge \1_{Y\setminus Z_\kappa} 
 T_i + 
dd^c (\psi_\kappa-\psi'_\kappa)  \wedge \1_{Z_\kappa} (T_+-T_-) + 
\sum_i d (\gamma_i \wedge dd^c \psi'_\kappa) \wedge \1_{Z_\kappa} T_i, 
\end{multline*}
which is of the form in the right hand side of \eqref{tennis} since $\psi_\kappa$ and
$\psi'_\kappa$ are smooth. 
Here we have used that set of closed positive currents with analytic
singularities is closed under multiplication with $\1_U$, where $U$ is
an constructible set and Remark \ref{stress}. 
The claim now follows by induction. 
\end{proof}

\section{Comparison with \cite{LRRS}, 
Proof of Theorem ~\ref{thmB}}\label{jamfora}

Assume that $h$ is a singular Griffiths positive (negative) metric on
a holomorphic vector
bundle $E\to X$ over a complex manifold $X$ of dimension $n$, such that 
that the degeneracy locus of $h$ 
is contained in a variety $V\subset X$. 
In \cite{LRRS} the first three authors together with Ruppenthal 
defined 
Chern and Segre currents, $c_k(E,h)$ and $s_k(E,h)$, for $k\leq \codim
V$. Let us briefly recall the construction. 
Locally, $h$ can be approximated by an increasing (decreasing) sequence
$h_\eps$ of Griffiths positive (negative) smooth metrics, see, e.g.,
\cite[Proposition~3.1]{BP} or \cite[Proposition~1.3]{R}. 
Theorem~1.11 in \cite{LRRS} asserts that the iterated 
limit\footnote{In \cite{LRRS} the limit is taken over certain subsequences
  of $h_\varepsilon$, but this is in fact not necessary; see the end
  of the proof of Proposition ~4.6 in \cite{LRRS}.}
\begin{equation}\label{iterated}
\lim_{\varepsilon_t \to 0} \cdots \lim_{\varepsilon_1 \to 0} s_{k_t}(E,h_{\varepsilon_t}) \w \cdots \w s_{k_1}(E,h_{\varepsilon_1})
\end{equation}
exists as a current and is independent of the choice of $h_\varepsilon$
for $k_{1}+\cdots + k_t\leq \codim V$; in particular, it follows that 
$s_{k_t}(E,h)\wedge\cdots\wedge s_{k_1}(E,h)$, locally given as
\eqref{iterated}, defines a global current on $X$. 
Moreover, the Chern currents $c_k(E,h)$, defined from
$s_{k_t}(E,h)\wedge\cdots\wedge s_{k_1}(E,h)$ analogously to \eqref{koko}, and the Segre currents $s_k(E,h)$ coincide with
the corresponding Chern and Segre forms where $h$ is smooth, and are
in the classes $c_k(E)$ and $s_k(E)$, respectively, 
when $X$ is compact.

Assume that $h$ is Griffiths positive and let $\varphi_\eps$ be the smooth metric
on $\pi:\P(E)\to X$ induced by $h_\eps$. Then $\varphi_\eps$ is a sequence of smooth positive metrics on $\Ok_{\P(E)}(1)$ decreasing to
$\varphi$. Let $\omega_\eps$ be the first Chern form
of $\varphi_\eps$. Then $s_{k_t}(E,h)\wedge\cdots\wedge
s_{k_1}(E,h)$ satisfies the following recursion for $t\geq 0$:
\begin{equation}\label{slowgold}
s_{k_t}(E,h)\wedge\cdots\wedge s_{k_1}(E,h) 
=
\lim_{\eps  \to 0} (-1)^{k_t}\pi_*(\omega_{\eps}^{k_t+r-1}) \wedge
s_{k_{t-1}}(E,h)\wedge\cdots\wedge s_{k_1}(E,h).
\end{equation}

\begin{remark}\label{konstmuseum}
Assume that $h$ is Griffiths positive. 
Let us use
the notation from Section ~\ref{plommon}, and denote the sequences of positive metrics on $Y$
induced by $h_{\eps}$ by $\varphi_{j,\eps}$. 
Moreover assume that we are outside the degeneracy locus of $h$. Then the
induced $\varphi_j$ are locally bounded and thus, by Proposition
~\ref{addis}, 
\begin{multline*}
(-1)^k s_{k_t}(E,h)\wedge\cdots\wedge s_{k_1}(E,h) 
=
\lim_{\varepsilon_t \to 0} \cdots \lim_{\varepsilon_1 \to 0} \
\pi_* \big ( (dd^c\varphi_{t,\eps_t})^{k_t+r-1}\wedge \cdots \wedge
(dd^c\varphi_{1,\eps_1})^{k_1+r-1} \big ) 
\\=
\pi_* \big ( (dd^c\varphi_{t})^{k_t+r-1}\wedge \cdots \wedge
(dd^c\varphi_{1})^{k_1+r-1} \big ).
\end{multline*}
\end{remark}

To prove Theorem ~\ref{thmB} we need to recall some auxiliary results from
\cite{LRRS}. 
First, following \cite{LRRS} we say that a smooth $(n-k,n-k)$-form
$\beta$ is a \emph{bump form} at a point $x\in X$ if it is strongly positive, and such that for some (or equivalently for any) K\"ahler form $\omega$ defined near $x$, there exists a constant $C > 0$ such that $C\omega^{n-k} \leq \beta$ as strongly positive forms in a neighborhood of $x$.

\begin{lma}\label{bump}
Let $V\subset X$ be a subvariety. 
Then for each $k\leq \codim V$ and each point $x\in V$, there
exists a bump form $\beta$ at $x$ of bidegree $(n-k,n-k)$ with
arbitrarily small support such that $dd^c\beta$ has support in
$X\setminus V$.
\end{lma}

\begin{proof}
We construct the bump form $\beta$ as in the proof of Lemma~4.3 in
\cite{LRRS} (with $k$ equal to $k+q$ in that proof).
By that proof, one may write $\beta$ as a sum of terms, such that each term
in some local coordinate system $(z',z'') \in \C^{n-k} \times \C^k$ is
of the form $\chi_1 \chi_2 \beta_0,$
where $\beta_0 = ~idz'_1 \wedge d\bar{z}'_1 \wedge \cdots \wedge i dz'_{n-k} \wedge d\bar z'_{n-k}$
and $\chi_1$ and $\chi_2$ are cutoff functions in the variables $z'$
and $z''$, respectively, such that 
$\chi_2$ is constant in a neighborhood of $(\supp \chi_1\chi_2) \cap
V$. 
It then suffices to prove that $d(\chi_1 \chi_2 \beta_0)$ has support in $X \setminus V$. This holds
since $\beta_0$ has full degree in the $z'$-variables so that
$d(\chi_1 \chi_2 \beta_0) = \chi_1 d\chi_2 \wedge \beta_0$, and
$\chi_1 d\chi_2$ has support in the set 
where $\chi_1 \chi_2\not\equiv 0$ and $\chi_2$ is not constant, which is contained in $X \setminus V$.
\end{proof}

The next result is Lemma~4.5 in \cite{LRRS}. 
\begin{lma}
\label{flopp}
Let $S$ and $T$ be two closed, positive $(k,k)$-currents on $X$ such that $S = T$ outside a subvariety $V$ with $\codim V\geq k$, and assume that for each point of $x \in V$, there exists an $(n-k,n-k)$ bump form $\beta$ at $x$ with arbitrarily small support such that
\begin{equation*}
    \int_X S \wedge \beta = \int_X T \wedge \beta.
\end{equation*}
Then $S = T$ everywhere.
\end{lma}

\begin{proof}[Proof of Theorem~\ref{thmB}]

We will proceed by induction. 
Let us first choose $t\geq 2$ and assume that we have proved that 
\begin{equation}\label{dusa}
s_{k_{t-1}}(E,h,\theta)\wedge\cdots\wedge s_{k_{1}}(E,h,\theta)
=
s_{k_{t-1}}(E,h)\wedge\cdots\wedge s_{k_{1}}(E,h). 
\end{equation}
Let us use the notation from Section ~\ref{plommon}.  
Moreover, let $Y'$ be the fiber product 
\[
Y'=\P(E_{t-1})\times_X\cdots \times_X \P(E_1),
\]
with projections $\varpi'_j: Y'\to \P(E_j)$, $\pi': Y'\to X$, and
$p:Y\to Y'$. Then $Y=\P(E_t) \times_X Y'$. 

Let $\varphi'_j$ denote the pullback metric $(\varpi'_j)^*\tilde
\varphi_j$ on $L'_j:=(\varpi'_j)^* L_j$ and let
$\theta'_j=(\varpi'_j)^*\tilde\theta_j$. 
Let $\tilde \varphi_\epsilon$ denote the metric on $\widetilde L_t$ induced
by 
$h_\eps$, let $\varphi_\eps$ denote
the pullback $\varpi_t^*\tilde\varphi_\eps$ to $Y$, and let $\tilde\omega_\eps$ and $\omega_\eps$
denote the corresponding first Chern forms. 
Let 
\[
\mu' = 
[dd^c\varphi'_{t-1}]^{k_{t-1}+r-1}_{\theta'_{t-1}}\wedge \cdots \wedge
[dd^c\varphi'_1]^{k_1+r-1}_{\theta'_1} 
\]
and let $\mu=p^*\mu'$; by regularization 
\[
\mu=[dd^c\varphi_{t-1}]^{k_{t-1}+r-1}_{\theta_{t-1}}\wedge \cdots \wedge
[dd^c\varphi_1]^{k_1+r-1}_{\theta_1}. 
\]

Now, using the induction hypothesis
\eqref{dusa} and Lemma ~\ref{forwards} we can rewrite \eqref{slowgold} as 
\begin{equation*}
s_{k_t}(E,h)\wedge\cdots\wedge s_{k_1}(E,h) 
=
(-1)^k \lim_{\eps  \to 0} (\pi_t)_* \tilde\omega_{\eps}^{k_t+r-1} \wedge
\pi'_* \mu'
=
(-1)^k \lim_{\eps  \to 0} \pi_*(\omega_{\eps}^{k_t+r-1} \wedge \mu). 
\end{equation*}
Moreover, 
\[
s_{k_t}(E,h, \theta)\wedge\cdots\wedge s_{k_1}(E,h, \theta) 
=
(-1)^k\pi_* \big ( [dd^c\varphi_t]^{k_t+r-1}_{\theta_t} \wedge \mu \big ).
\]

Since $k\leq \codim V$, by Lemma ~\ref{bump}, for each $x\in V$ there is a bump
form $\beta$ at $x$ of bidegree $(n-k,n-k)$ with arbitrarily small support such
that $dd^c\beta$ vanishes in a neighborhood of $V$. 
Note that $\pi_t(L(\tilde \varphi_t))\subset V$ in view of Lemma
~\ref{annandag}. It follows that 
\[
Z_t =\varpi_t^{-1} L(\tilde \varphi_t)\subset 
\varpi_t^{-1} \pi_t^{-1} V =
\pi^{-1} V 
\]
and thus $dd^c\pi^*\beta$ vanishes in a neighborhood of
$Z_t\subset Y$. 
Hence, by Lemma ~\ref{medusa} (applied to $T=\mu$) 
\begin{multline*}
\int_X s_{k_t}(E,h, \theta)\wedge\cdots\wedge s_{k_1}(E,h, \theta)
\wedge \beta
=
(-1)^k \int_Y
[dd^c\varphi_t]^{k_t+r-1}_{\theta_t} \wedge \mu \wedge \pi^* \beta
=\\
(-1)^k\lim_{\eps  \to 0} \int_Y  \omega_{\eps}^{k_t+r-1} \wedge \mu \wedge \pi^* \beta
=
\int_X s_{k_t}(E,h)\wedge\cdots\wedge s_{k_1}(E,h)
\wedge \beta. 
\end{multline*}
In view of Remarks ~\ref{bibliotek} and ~\ref{konstmuseum},
\eqref{jason} holds outside $V$, and thus by Lemma ~\ref{flopp} it holds
everywhere.

\smallskip 
It remains to prove \eqref{jason} for $t=1$. This follows, in fact, by
an easier version of the argument above. If $\beta$ is a bump form as
above, then by Lemma ~\ref{medusa} (with $T=1$)
\begin{multline*}
\int_X s_{k_1}(E,h, \theta)
\wedge \beta
=
(-1)^{k_1}\int_{\P(E_1)}
[dd^c\varphi_1]^{k_1+r-1}_{\theta} \wedge \pi_1^* \beta
=\\
(-1)^{k_1}\lim_{\eps  \to 0} \int_{\P(E_1)} \omega_{\eps}^{k_1+r-1} \wedge \pi_1^* \beta
=
\int_X s_{k_1}(E,h)
\wedge \beta
\end{multline*}
and again \eqref{jason} follows from Lemma ~\ref{flopp}. 
\end{proof}

\section{Remarks and examples}\label{losi}
Let us start by discussing the uniqueness of the Chern and Segre
currents. 
Assume that $X$ is a complex manifold and that $V\subset X$ is a subvariety of pure 
codimension $p$. 
Moreover assume that 
$T_1$ and $T_2$ are closed positive $(p,p)$-currents on $X$ that
coincide outside $V$, and that the
Lelong numbers of $T_1$ and $T_2$ coincide at each $x\in V$. We claim
that then $T_1=T_2$. 
Indeed, since $T:=T_1-T_2$ is a closed normal $(p,p)$-current
with support on $V$ it follows that $T=\sum a_j [V_j]$, where $V_j$ are the
irreducible components of $V$, see, e.g.,
\cite[Corollary~III.2.14]{Dem}. Next, by assumption the Lelong number of $T$ at each
point in $V$ is zero and therefore $a_j=0$ for each $j$. 
If $T_1$ and $T_2$ are closed positive $(k,k)$-currents that coincide
outside $V$, where $k<p$, then $\1_VT_j$ vanishes for
$j=1,2$ by the dimension principle, and hence $T_1=T_2$.

Now assume that we are in the situation of Theorem \ref{thmA} and that
$L(\log \det h^*)\subset V$. Then by Remark
\ref{bibliotek}, $c_k(E, h, \theta)$ and $s_k(E,h,\theta)$ are
independent of $\theta$ outside $V$. Since 
they are of bidegree $(k,k)$ and (locally) differences of
closed positive currents it follows in view of \eqref{tredje} that
they are independent of $\theta$ for $k\leq p$. 
Note that if $h$ is smooth outside $V$ then $c_k(E, h, \theta)$ and $s_k(E, h,
\theta)$ are uniquely determined by the condition \eqref{andra} for
$k<p$.

On the other hand if $k>p$, let $\alpha$ and $\beta$ be real smooth
forms of bidegree $(k-p, k-p)$ such that $\alpha-\beta$ is exact. 
Then $(\alpha-\beta)\wedge [V]\neq 0$
has zero Lelong numbers everywhere, is cohomologous to zero, and vanishes
outside $V$. 
Thus there is no reason to expect $c_k(E,h,\theta)$ and $s_k(E,h,\theta)$ to be
independent of $\theta$ for $k>p$ in general.

\smallskip

Let us consider some simple examples, where we can compute the Segre and
Chern currents explicitly.

\begin{ex}\label{kass}
Let $L\to X$ be a line bundle and $e^{-\varphi}$ a Griffiths
positive metric with analytic singularities. Then $\Ok_{\P(L)}(1)=L$ and $e^{_-\varphi}=h$, and thus 
\begin{equation*}
s_k(L,h,\theta)=[dd^c\varphi]^k_\theta
=(dd^c \varphi)^k+\sum_{\ell=0}^{k-1} \theta^{k-\ell}\wedge \1_Z
(dd^c\varphi)^\ell, 
\end{equation*}
where $Z$ is the unbounded locus of $\varphi$. 
Classically, by the Bedford-Taylor-Demailly theory, for a general
$\varphi$, $(dd^c\varphi)^k$
is well-defined only for $k=1$; if $\varphi$ has analytic
singularities it is well-defined for 
$k\leq \codim Z=:p$. 

In fact, it is not hard to find
examples of psh functions $u$ with analytic singularities and
sequences $u^{(\iota)}$ of psh functions decreasing to $u$
where the corresponding sequences $(dd^c u^{(\iota)})^k$ converge to different positive
currents for $k>p$, see, e.g., \cite[Example~3.2]{ABW}. In particular,
this implies that the construction in \cite{LRRS} cannot extend to
$k>p$ in general. 
\end{ex}

\begin{ex}\label{linjalen}
Let $X=\P^n$, $L=\Ok_{\P^n}(1)$, and $h=e^{-\varphi}$, where
$\varphi=\log |s|^2$ and $s$ is
a non-trivial global holomorphic section of $L$, cf.\ Example ~\ref{fagel}. Then the unbounded
locus of $\varphi$ is the hyperplane $Z=\{s=0\}\subset \P^n$ and thus $(dd^c \varphi)^m$
is defined classically by the Bedford-Taylor-Demailly theory only for
$m=1$, cf.\ Example \ref{kass}. By the Poincar\'e-Lelong formula, 
$dd^c \varphi = [s=0]=[Z]$, cf.\
\cite[Example~2.2]{Dem2}.  It follows that 
\[
(dd^c \varphi)^2 = dd^c (\varphi \1_{X\setminus Z}dd^c\varphi) =0
\]
and thus $(dd^c \varphi)^m=0$ for all $m>1$. 
Hence, if $\theta$ is the first Chern form of a smooth metric on
$L$, then 
\[
[dd^c \varphi]^m_\theta=\theta^{m-1}\wedge [s=0]. 
\]

Since $L$ is a line bundle, $\P(L) = X$ and
$(\Ok_{\P(L)}(1),h) = (L,e^{-\varphi})$. 
Moreover, 
$Y= X$ and $\varphi_j=\varphi$
and $\theta_j=\theta$ for each $j$. Thus 
\begin{equation}\label{orange}
s_{k_t}(L,h,\theta) \wedge \cdots \wedge
s_{k_1}(L,h,\theta)
=
(-1)^k [dd^c\varphi]^k_\theta=
(-1)^k ~\theta^{k-1}\wedge [s=0]. 
\end{equation}
In particular, \eqref{orange} depends on 
$\theta$ as soon as $k>1$. In this case it is easy to see that the
Lelong numbers are independent of $\theta$, since $\theta$ is smooth. 

Note that $s_k(E,h)$ and $c_k(E,h)$ are
well-defined in the classical or \cite{LRRS} sense only for
$k\leq 1$; it holds that 
\[
c_1(E,h)=-s_1(E,h)=dd^c \varphi =[s=0].
\]
\end{ex}

\begin{ex}\label{lukt}
Let $E$ be a trivial rank $2$ bundle over $X=\C^2$ with coordinates
$x=(x_1,x_2)$ and let $h$ be the
singular metric $h = e^{-\log |x|^2} \Id$. 
In view of Section ~\ref{jensen}, in the open set $\U_i$ the induced metric
$\varphi$ on $\Ok_{\P(E)}(1)$ is given by 
\[
\varphi_i(x,[\xi])=\log|x|^2+\log |\xi/\xi_i|^2.
\]
In particular, it follows that $h$ is Griffiths positive with analytic
singularities. 

Since the unbounded locus of $\varphi$, $Z=\{x=0\}$, has codimension $2$
in $\P(E)$, $(dd^c\varphi)^m$ is classically well-defined for $m\leq
2$. Note that 
$dd^c\varphi = dd^c \log |x|^2 +\omega_{\text{FS}},$
where $\omega_{\text{FS}}$ is the Fubini-Study metric on the fibers
$\pi^{-1}(x)\cong \P^{1}_\xi$, 
and 
\begin{equation}\label{kaffegodis}
(dd^c\varphi)^2 = 
(dd^c \log |x|^2)^2 + 2\omega_{\text{FS}} \wedge dd^c \log |x|^2=
[x=0] + 2\omega_{\text{FS}} \wedge dd^c \log |x|^2 
\end{equation} 
since $\omega_{\text{FS}}^2$ vanishes for degree reasons. 
It follows that  
\begin{multline*}
(dd^c\varphi)^3 = dd^c (\varphi \1_{\P(E)\setminus Z}(dd^c\varphi)^2)
=
dd^c(\varphi ~ 2 \omega_{\text{FS}}\wedge dd^c \log |x|^2)
=\\
(\omega_{\text{FS}} +  dd^c \log |x|^2 )  \wedge (2 \omega_{\text{FS}}\wedge
dd^c \log |x|^2)
=
2\omega_{\text{FS}}\wedge [x=0],
\end{multline*}
where we have again used that $\omega_{\text{FS}}^2=0$. 
Thus if $\theta$ is the first Chern
form of a smooth metric on $\Ok_{\P(E)}(1)$, 
\[
[dd^c\varphi]^3_\theta = (dd^c \varphi)^3 +\theta\wedge \1_Z (dd^c
\varphi)^2 + \theta^2\wedge \1_Z dd^c \varphi
=
2\omega_{\text{FS}} \wedge [x=0] + \theta \wedge [x=0], 
\]
where the last term in the middle expression vanishes by the dimension
principle since $\codim Z=2$. 
Therefore 
\[
s_2(E,h,\theta)=\pi_*[dd^c\varphi]^3_\theta = 3 [0]. 
\]

In view of \eqref{kaffegodis}, $c_1(E,h, \theta)= - s_1(E,h, \theta)
=2 dd^c\log|x|^2$, and thus by \eqref{koko} we get that $c_2(E,
h, \theta)=[0]$. 
\end{ex}

A naive attempt would be to define Segre currents as the pushforward of $(dd^c\varphi)^{k+r-1}$ instead of
$[dd^c\varphi]_\theta^{k+r-1}$. Since $(dd^c\varphi)^m$ coincides with the
classical Bedford-Taylor-Demailly Monge-Amp\`ere product where $\varphi$ is locally
bounded, $\pi_*(dd^c\varphi)^{k+r-1}$ coincides with $s_k(E,h)$ where $h$
is smooth, cf.\ Lemma \ref{annandag}. 
Example \ref{linjalen}, however, shows that the current $\pi_*(dd^c
\varphi)^{k+r-1}$ is not in $s_k(E)$ in general; in that example
$(dd^c\varphi)^m=0$ for $m>1$, whereas $s_k(E)\neq 0$ for $0\leq k\leq
n$. Moreover, Example \ref{lukt} shows that the Lelong number of
$\pi_*(dd^c\varphi)^{k+r-1}$ is not equal to the Lelong number of
$s_k(E,h,\theta)$ in general. Indeed, 
note that in that example $\pi_*(dd^c\varphi)^3$ equals $2[0]$ and thus has Lelong number $2$ at the
origin, whereas the Lelong number at the origin of $s_2(E,h,\theta)$ is $3$. 

\smallskip 

The following example shows that the products \eqref{korsbar} of Segre currents are not
commutative in general.

\begin{ex}\label{martinex}
Let $X$ be the unit ball in $\C^3$ with coordinates $x=(z,\zeta_1,\zeta_2)$, and let
$E=X\times \C^2\to
X$ be the trivial vector bundle of rank $2$. 
Let $h$ be the singular hermitian metric on $E$ whose dual metric
$h^*$ on $E^*$ is given by the matrix 
$\begin{bmatrix} 0 & 0 \\ 0 & |z|^2 \end{bmatrix}$. Then in view
of Section ~\ref{jensen}, the induced metric $\varphi$ on
$\Ok_{\P(E)}(1)$ is given by 
$\varphi_1(x,[\xi])=\log|z|^2+\log |\xi_2/\xi_1|^2$ and
$\varphi_2=\log|z|^2$ in $\U_1$ and $\U_2$, respectively. It follows
that 
\[dd^c \varphi = dd^c(\log |z|^2 +\log|\xi_2|^2)=[Z]+[W],
\]
where $Z=\{z=0\}$ and $W=\{\xi_2=0\}$. 

Moreover let $g$ be the smooth metric on $E$ given by the matrix 
$\begin{bmatrix} 1 & |\zeta|^2 \\ |\zeta|^2 & 1\end{bmatrix}$. 
A computation yields that the
curvature form at $\zeta=0$ 
is $\Theta^g|_{\zeta=0}=\begin{bmatrix} 0 & \dbar \partial |\zeta|^2 \\
  \dbar \partial |\zeta|^2 & 0\end{bmatrix}$ so that 
$\frac{i}{2\pi}\Theta^g|_{\zeta=0}=-\begin{bmatrix} 0 & dd^c |\zeta|^2 \\
  dd^c |\zeta|^2 & 0\end{bmatrix}$. 
Thus at $\zeta=0$, in view of \eqref{identitet}, 
\begin{equation*}
s_1(E,g) = - c_1(E,g) =0, ~~~ c_2(E,g)= - (dd^c |\zeta|^2)^2, ~~~ s_2(E,g)= c_1(E,g)^2-c_2(E,g)=(dd^c |\zeta|^2)^2.
\end{equation*}

Let $\theta$ be the first Chern form of the smooth metric $\psi$ on $\Ok_{\P(E)}(1)$
induced by $g$. Then at $(x,[\xi])\in\P(E)$ 
\[
\theta = dd^c\psi = \omega^{g^*}_{\text{FS}}-
\frac{i}{2\pi|\xi|_g}\Theta^{g^*}_{\xi\overline \xi},
\] 
where $\Theta^{g^*}$ is the curvature form on $E_x^*$ and
$\omega^{g^*}_{\text{FS}}$ is the induced Fubini-Study metric on the fiber
$\pi^{-1}(x)=\P(E_x^*)\cong\P^1$, 
see, e.g., the beginning of the proof of Proposition~3.1 in \cite{G}
or the beginning of Section ~2 in \cite{Div}. 

Note that at $\zeta=0$, $g$ is just the standard Euclidean metric on $\C^2$, so that $\omega^{g^*}_{\text{FS}}$ is just the
standard Fubini-Study metric $\omega_{\text{FS}}$ on $\P^1$. Moreover,
$\Theta^{g^*}|_{\zeta=0}=-(\Theta^g)^T|_{\zeta=0}=-\begin{bmatrix} 0 & \dbar \partial |\zeta|^2 \\
  \dbar \partial |\zeta|^2 & 0\end{bmatrix}$, where $T$ denotes transpose. 
In particular, for $(x,[\xi])$ such that $\zeta =0$ and $\xi_2=0$,
$\Theta^{g^*}_{\xi\overline \xi}=0$. 
Hence at $\zeta=0$, 
\begin{equation}\label{nina}
\theta\wedge [W]=\omega_{\text{FS}}\wedge [W]=0,
\end{equation}
where the last equality follows for degree reasons. 
Therefore, for $m>1$, noting that $(dd^c\varphi)^m=0$,
$
[dd^c\varphi]^m_\theta=\theta^{m-1}\wedge \big
([Z]+[W])=\theta^{m-1}\wedge [Z]
$
at $\zeta=0$. 
More generally, let $E_1$ and $E_2$ be copies of $E$ and $Y=\P(E_1)\times_X\P(E_2)$,
and let us use the notation from Section ~\ref{plommon}. Then a
computation using \eqref{nina}, yields that for $m_1, m_2>1$, at $\zeta =0$, 
\[
[dd^c\varphi_2]^{m_2}_{\theta_2}\wedge
[dd^c\varphi_1]^{m_1}_{\theta_1}=
\theta_2^{m_2}\wedge \theta_1^{m_1-1}\wedge [Z]. 
\]
In view of \eqref{daghem2}, \eqref{rodhake}, and
\eqref{fiskmoll} 
it follows that 
\[
\pi_*\big (\theta_2^{k_2+1}\wedge \theta_1^{k_1+1}\wedge [Z]\big )
=(-1)^{k_1+k_2}s_{k_2}(E,g)\wedge s_{k_1}(E,g)\wedge [Z].
\]
Hence at $\zeta=0$
\begin{multline*}
s_1(E,h,\theta)\wedge s_2(E,h,\theta) = 
-\pi_*\big ( 
[dd^c\varphi_2]^{2}_{\theta_2}\wedge
[dd^c\varphi_1]^{3}_{\theta_1}
\big ) = \\
-\pi_*\big (\theta_2^{2}\wedge \theta_1^{2}\wedge [Z]\big )=
-s_{1}(E,g)\wedge s_{1}(E,g)\wedge [Z] = 0, 
\end{multline*}
and similarly 
\[
s_2(E,h,\theta)\wedge s_1(E,h,\theta) = 
-s_{2}(E,g)\wedge s_{0}(E,g)\wedge [Z] = -(dd^c|\zeta|^2)^2\wedge
[Z]\neq 0. 
\]
Thus $s_1(E,h,\theta)\wedge s_2(E,h,\theta) \neq  s_2(E,h,\theta)\wedge
s_1(E,h,\theta)$ in this case. 
\end{ex}

\end{document}